\documentclass{amsart}

\usepackage{amsmath}
\usepackage{amssymb}
\usepackage{amsthm}
\usepackage{enumerate}

\usepackage{color}

\newtheorem{thm}{Theorem}[section]
\newtheorem{prop}[thm]{Proposition}
\newtheorem{cor}[thm]{Corollary}
\newtheorem{lem}[thm]{Lemma}

\newtheorem{cond}[thm]{Condition}
\newtheorem{fact}[thm]{Fact}

\theoremstyle{definition}

\newtheorem{dfn}[thm]{Definition}

\newtheorem{rmk}[thm]{Remark}

\numberwithin{equation}{section}

\newcommand{\cB}{\mathcal{B}}

\newcommand{\cI}{\mathcal{I}}

\newcommand{\cM}{\mathcal{M}}
\newcommand{\cN}{\mathcal{N}}
\newcommand{\cS}{\mathcal{S}}
\newcommand{\cT}{\mathcal{T}}

\newcommand{\id}{\textrm{id}}

\newcommand{\BMO}{{\rm bmo}}
\newcommand{\sfG}{\mathsf{G}}

\newcommand{\boldA}{\textbf{A}}
\newcommand{\boldB}{\textbf{B}}
\newcommand{\boldC}{\textbf{C}}

\newcommand{\BigBMO}{ {\rm BMO} }

\begin{document}

\title[Vector valued BMO-estimates and perturbations of commutators]{BMO-estimates for non-commutative vector valued Lipschitz functions}

\date{}

\author{M. Caspers, M. Junge, F. Sukochev, D. Zanin}

\date{\today, {\it MSC2000}: 47B10, 47L20, 47A30.  The work of  FS and DZ is supported by the ARC}

\address{M. Caspers, TU Delft, DIAM, Analysis, Van Mourik Broekmanweg 6, 2628 XE Delft, The Netherlands}
\email{m.p.t.caspers@tudelft.nl}

\address{M. Junge, Department of Mathematics, University of Illinois, Urbana, IL 61801, USA}
\email{junge@math.uiuc.edu}

\address{F. Sukochev, D. Zanin, School of Mathematics and Statistics, UNSW, Kensington 2052, NSW, Australia}
\email{f.sukochev@unsw.edu.au}
\email{d.zanin@unsw.edu.au}

\begin{abstract}
We construct   Markov semi-groups $\cT$ and associated BMO-spaces on a finite von Neumann algebra $(\cM, \tau)$ and obtain results for perturbations of commutators and non-commutative Lipschitz estimates. In particular, we prove that for any $A \in \cM$  self-adjoint and $f: \mathbb{R} \rightarrow \mathbb{R}$  Lipschitz there is a   Markov semi-group $\cT$ such that for $x \in \cM$,
\[
\Vert [f(A), x] \Vert_{\BMO(\cM, \cT)} \leq c_{abs}   \Vert f' \Vert_\infty \Vert [A, x] \Vert_\infty.
\]
We obtain an analogue of this result for more general von Neumann valued-functions $f: \mathbb{R}^n \rightarrow \cN$ by imposing  H\"ormander-Mikhlin type assumptions on $f$.

 In establishing these result we show that Markov dilations of Markov semi-groups have certain automatic continuity properties. We also show that Markov semi-groups of double operator integrals admit (standard and reversed) Markov dilations.
 \end{abstract}

\maketitle

\section{Introduction}

Non-commutative Lipschitz properties of functions have been studied for a long time and go back at least to the work of M.G. Krein \cite{Krein}. One question raised in \cite{Krein} in this direction is whether every Lipschitz function $f: \mathbb{R} \rightarrow \mathbb{C}$ is also a non-commutative Lipschitz function in the sense that the mapping
\begin{equation}\label{Eqn=NCLip}
B(H)_{sa} \rightarrow B(H): A \mapsto f(A),
\end{equation}
is Lipschitz. Here $B(H)_{sa}$ is the self-adjoint part of the bounded operators on a Hilbert space $B(H)$.   In its original statement,  Krein's question has a negative answer as was shown in \cite{Far67}, \cite{Far68}, \cite{Far72}. In fact already for $f$ the aboslute value map the statement fails  \cite{Davies}, \cite{Kato}.  Only after imposing additional smoothness/differentiability properties on $f$  the mapping \eqref{Eqn=NCLip} is Lipschitz.   Indeed,  in \cite{BirmanSolomyak}, \cite{BirmanSolomyak2} Birman and Solomyak showed that for $f' \in {\rm Lip}_{\varepsilon}(\mathbb{R})\cap L^p(\mathbb{R})\cap L_\infty(\mathbb{R})$ with $\varepsilon > 0, p\ge 1$  we have that \eqref{Eqn=NCLip} is Lipschitz.  The result was improved on by Peller in \cite{PellerHankel}, \cite{Peller2} who showed that it suffices to take $f$ in the Besov space $B^{1}_{\infty 1}$, see  \cite{Grafakos} for Besov spaces. 

Krein's question can be altered by replacing the uniform operator norms in \eqref{Eqn=NCLip} by non-commutative $L_p$-norms with $1 < p < \infty$ associated with the Schatten-von Neumann classes $\cS_p$. In this case a complete answer to the non-commutative differentiability properties of \eqref{Eqn=NCLip} was found \cite{PotapovSukochevActa}, namely any Lipschitz function is a non-commutative Lipschitz function in the sense that there is a constant $c_p$ such that for any self-adjoint operators $A, B \in \mathcal{S}_p$ we have,
\[
\Vert f(A) - f(B) \Vert_p \leq c_p\|f'\|_{\infty} \Vert A - B \Vert_p.
\]
The constant $c_p$ grows to $\infty$ if either $p \rightarrow 1$ or $p \rightarrow \infty$. In fact the asymptotic  behaviour was found in \cite{CMPS} (see also \cite{CPSZ}) where it was shown that asymptotically $c_p \simeq p^2 (p-1)^{-1}$.

In this paper we start the investigation of perturbation of commutators and non-commutative Lipschitz functions from two new view points: BMO-spaces and vector valued estimates.

\vspace{0.3cm}

We use the theory of BMO-spaces to obtain `end-point estimates' of Krein's problem. The optimal behaviour for the constant $c_p$ hints towards the existence of such an end-point estimate but so far the proof was not obtained. In this context we use the theory of semi-group BMO-spaces, in the commutative case extensively studied by  e.g. \cite{StroockVar}, \cite{Varopoulos}, and much more recently in \cite{DY5}, \cite{DY5b}. For non-commutative BMO-spaces the theory was developped in \cite{JungeMei}, see also \cite{JMPReview}.

 BMO-spaces depend on the choice of a semi-group. This is just as for  other definitions of BMO, which depend on the filtration of a von Neumann algebra or in the classical setting the choice of cubes/shapes over which means are taken. This choice gives a flexibility in finding the appropriate BMO-space for Krein's problem.  In the current paper we introduce a natural BMO-space to resolve such problems in perturbation theory. In particular, we prove the result announced in the abstract. Our main theorem which makes this all work, proved in Section \ref{Sect=ComBMO}, yields as follows.

\begin{thm} \label{Thm=Intro} Let $(\cM, \tau)$ be a finite von Neumann algebra and let $f:\mathbb{R}\to\mathbb{R}$ be Lipschitz with $\Vert f ' \Vert_\infty \leq 1$. For every $A=A^*\in\mathcal{M},$
\begin{enumerate}[{\rm (i)}]
\item  The semi-group of double operator integrals $\mathcal{I}^A = ( \cI_{e^{-tF}}^A )_{t\geq0}$ with symbol
$$F(\lambda,\mu)=|\lambda-\mu|^2+|f(\lambda)-f(\mu)|^2,\quad \lambda,\mu\in\mathbb{R},$$
is Markov (i.e. a  strongly continuous semi-group of trace preserving unital completely positive maps);
\item  The double operator integral  $\cI_{f^{[1]}}^A$ with $f^{[1]}$ the divided difference of $f$ maps $\mathcal{M}$ to $\BMO_{\cI^A}(\cM)$ and its norm is bounded by an absolute constant $c_{abs}$.
\end{enumerate}
\end{thm}

As a corollary of Theorem \ref{Thm=Intro} we retrieve many existing results in perturbation theory, in particular the ones from \cite{Kato}, \cite{Davies},  \cite{Kosaki}, \cite{DDPS1}, \cite{DDPS2},  \cite{PotapovSukochevActa}, \cite{CMPS}, and partly \cite{KPSS}. We also retrieve the optimal estimates in case $p = \infty$ for finite dimensional Schatten classes in Theorem \ref{Thm=LogEstimate}, see \cite{AleksandrovPeller}. Together with the  weak $(1,1)$ estimate of \cite{CPSZ} (see also \cite{NazarovPeller}), which is complementary to our paper, they complete the study of the end-point estimates.   At the same time, we emphasize that our results do not cover the case of infinite von Neumann algebras, due to the fact that BMO-spaces, even in the case  $\cM=B(H)$, are not realized as spaces of operators. On the other hands, a lot of techniques and proofs developed in this paper continue to hold for general semifinite von Neumann algebras almost verbatim (see also Section \ref{Sect=Heat}) and this is a cause for careful optimism that our approach can be extended to the latter case as well. 

To apply Theorem \ref{Thm=Intro} and obtain these corollaries we shall further develop the theory of Markov dilations and we obtain some results of independent interest. In particular we show that Markov semi-groups can be studied through their discrete subsemi-groups and get automatic continuity of a Markov dilation. The following is proved in Theorem \ref{discretization theorem}.

\begin{thm}\label{Thm=IntroII}  Let $(\cM, \tau)$ be a finite von Neumann algebra.  Let $\mathcal{T}=(T_t)_{t\geq0}$ be a Markov semi-group. If $(T_t)_{t\in\epsilon\mathbb{N}_{\geq 0}}$ admits a standard (resp. reversed) Markov dilation for every $\epsilon>0,$ then also $\mathcal{T}$ admits a standard (resp. reversed) Markov dilation. Moreover, the dilation has continuous path.
\end{thm}

We then apply this to Markov semi-groups of double operator integrals and through Ricard's results \cite{RicardDilation} on dilations of Schur multipliers we prove that they also admit a (standard and reversed) Markov dilation.

 \vspace{0.3cm}

In the final part of the paper, Section \ref{Sect=Vector}, we initiate the study of vector-valued Lipschitz functions (in fact von Neumann   algebra  valued to be precise). As we show in Corollary \ref{Cor=SemiCircle}  Khintchine type inequalities and free probability estimates can be recasted in terms of perturbations of vector valued commutators. Section \ref{Sect=Vector} is strongly based on non-commutative Calder\'on-Zygmund theory as developed in \cite{JMP}; in particular we obtain our results through the non-commutative H\"ormander-Mikhlin theorem of \cite{JMP}.

 \vspace{0.3cm}

\noindent {\bf Structure of the paper.} Section \ref{Sect=Prelim} recalls all preliminaries and settles notation. Section \ref{Sect=Disc} proves our discretization result for reversed Markov dilations, i.e. Theorem \ref{Thm=IntroII}. Then in Section \ref{Sect=DOIDilate} we show that Markov semi-groups of double operator integrals admit a reversed Markov dilation. We collect the corresponding results on standard Markov dilations in Section \ref{Sect=ReversedDilateFix}; these results are not used in this paper but we believe they are of independent interest and state them for convenience of the reader. Section  \ref{Sect=ComBMO} proves Theorem \ref{Thm=IntroII} and we derive all its corollaries for perturbation theory in Section \ref{Sect=Conclusions}.
In Section \ref{Sect=VectTrans} and in Section \ref{Sect=Vector} we retrieve the von Neumann-valued Lipschitz estimates.

 \vspace{0.3cm}

 \noindent {\bf Acknowledgements.} We thank both anonymous referees for their careful reading and suggesting improvements to the manuscript.

\section{Preliminaries}\label{Sect=Prelim}

\subsection{General notations} For a multi-index $\alpha = (\alpha_1, \ldots, \alpha_n)$ we write $\vert \alpha \vert = \sum_{k=1}^n \alpha_k$.  For a finite von Neumann algebra $\cM$ with faithful normal trace $\tau$  we write $L_2(\cM)$ for the non-commutative $L_2$-space with respect to $\tau$. We let $\Omega_\tau = 1_{\cM} \in L_2(\cM)$ be the cyclic vector. We identifiy elements of $\cM$ as vectors in $L_2(\cM)$ if necessary. We write $L_p(\cM)$ for the non-commutative $L_p$-space, $1 \leq p < \infty$, associated with $\cM$ and $\tau$. It is the space of all closed densely defined operators $x$ affiliated with $\cM$ such that $\Vert x \Vert_p = \tau(\vert x \vert^p)^{1/p}$ is finite.  Naturally $\cM \subseteq L_p(\cM)$. We set $L_\infty(\cM) = \cM$. The $L_2$-topology on $\cM$ is then the topology of the norm $\Vert \: \Vert_2$.

\subsection{Non-commutative finite BMO-spaces}\label{Sect=BMO}
We recall the following from \cite{JungeMei}. Fix a finite von Neumann algebra $(\cM, \tau)$. We restrict ourselves here to the finite case in order to avoid several technicalities. We then treat the (non-finite) Euclidian case separately  in Section \ref{Sect=Heat}.

\begin{dfn}\label{Dfn=MarkovSemi} We say that a semi-group $\cT = (T_t)_{t \geq 0}$ of linear maps $\cM \rightarrow \cM$ is a Markov semi-group if:
\begin{enumerate}[{\rm (i)}]
\item\label{markovi} $T_t(1) = 1$ and $T_t$   completely positive for every $t\geq0;$
\item\label{markovii} for every $x,y \in \cM$ and for every $t \geq 0$ we have $\tau(x T_t(y)) = \tau(T_t(x) y);$
\item\label{markoviii} for every $x\in \cM,$ we have $t \mapsto T_t(x)$ is continuous in measure.
\end{enumerate}
\end{dfn}

Fix such a  Markov semi-group $\cT = (T_t)_{t \geq 0}$. By a standard interpolation argument for every $t \geq 0$ the map $T_t$ extends to a completely contractive map,
 \[
 T_t^{p}: L_p(\cM) \rightarrow L_p(\cM): x \mapsto T_t(x), \qquad \forall x \in \cM \subseteq L_p(\cM).
 \]
 We set
 \[
 \cM^\circ = \left\{  x \in  \cM  \mid \lim_{t \rightarrow \infty} T_t(x) = 0 \right\},
 \]
  where the limit is a  $\sigma$-weak limit. For  $1 \leq p < \infty$ we set by a norm limit,
 \[
 L_p^\circ(\cM) = \left\{  x \in L_p(\cM) \mid \lim_{t \rightarrow \infty} T_t^p(x) = 0 \right\}.
 \]
  It is a straightforward verification that $ L_p^\circ(\cM)$ is a Banach space in the induced norm. 
 For $x \in \cM^\circ$ we set the column BMO-norm,
 \begin{equation}
 \Vert x \Vert_{\BMO_{\cT}^c} = \sup_{t \geq 0} \Vert T_t(x^\ast x) - T_t(x)^\ast T_t(x) \Vert^{\frac{1}{2}}_\infty.
 \end{equation}
 Further set,
 \begin{equation}\label{Eqn=BMOother}
 \Vert x \Vert_{\BMO_{\cT}^r} = \Vert x^\ast \Vert_{\BMO_{\cT}^c}, \qquad \Vert x \Vert_{\BMO_{\cT}} = \max( \Vert x \Vert_{\BMO_{\cT}^r}, \Vert x^\ast \Vert_{\BMO_{\cT}^c}).
 \end{equation}
  We define $\BMO_{\cT} = \BMO(\cM, \cT)$ as the completion of the space of $x \in \cM^\circ$ with $\Vert x \Vert_{\BMO_{\cT}} < \infty$; it carries norm $\Vert \: \Vert_{\BMO_{\cT}}$.
  We  have contractive inclusions, see \cite[Lemma 3.6]{CaspersBMO},
 \[
 \cM^\circ \subseteq   \BMO(\cM, \cT) \subseteq L_1(\cM).
 \]
 This allows us to represent elements of $\BMO$ as concrete operators that are affiliated with $\cM$ and which are $L_1$ and in particular $\tau$-measurable; this is again a reason to prefer working in the finite setting.
 In particular $L_1(\cM)$ and $\BMO(\cM, \cT)$ form a compatible couple of Banach spaces.
 Also we impose the operator space structure,
 \[
M_n( \BMO(\cM, \cT)) =   \BMO(M_n \otimes \cM, \id_n \otimes \cT).
 \]

We will also make use of the following alternative BMO-norm. For $x \in \mathcal{M}^\circ$ we set
\[
\Vert x \Vert_{\BigBMO^c_{\mathcal{T}}} = \sup_{t \geq 0}  \Vert  T_t\left(  \vert x - T_t(x) \vert^2 \right)  \Vert_\infty^{\frac{1}{2}}.
\]
Then put $\Vert x \Vert_{\BigBMO^r_{\mathcal{T}}}   = \Vert x^\ast \Vert_{\BigBMO^c_{\mathcal{T}}}$ and $	\Vert x \Vert_{\BigBMO_{\mathcal{T}}} = \max( \Vert x \Vert_{\BigBMO^c_{\mathcal{T}}}, \Vert x \Vert_{\BigBMO^r_{\mathcal{T}}}  )$.  The completion of $\mathcal{M}$ for $\Vert \:\cdot \: \Vert_{\BigBMO_{\mathcal{T}}}$ is then defined as $\BigBMO_{\mathcal{T}} := \BigBMO(\mathcal{M}, \mathcal{T} )$.   We observe that $L_1(\cM)$ and $\BigBMO(\cM, \cT)$ form a compatible couple of Banach spaces.

For later use, we record the following lemma here.

 \begin{lem}\label{approximate markov fact} Let $\cM$ be a finite von Neumann algebra. Let $\mathcal{T}=(T_t)_{t\geq0}$ and $\mathcal{T}^l=(T^l_t)_{t\geq0}, l \in \mathbb{N}$ be semigroups on $\cM.$ Suppose $\mathcal{T}$ is continuous  in measure in the sense of Definition \ref{Dfn=MarkovSemi} (\ref{markoviii}). If $\mathcal{T}_l$ is Markov for each $l$ and if $T^l_t(x)\to T_t(x)$ in measure for $x\in\cM$ as $l\to\infty,$ then $\mathcal{T}$ is Markov.
 \end{lem}

  \subsection{Markov dilations}  Recall the following definition from \cite[Page 717]{JungeMei}.

 \begin{dfn}
We say that a Markov semi-group $\mathcal{T} = (T_t)_{t \geq 0}$ on a finite von Neumann algebra $(\cM, \tau)$ admits a standard Markov dilation if there exist:
\begin{enumerate}[{\rm (i)}]
\item a finite von Neumann algebra $(\mathcal{B}, \tau_{\cB})$;
\item an increasing filtration $\mathcal{B}_s, s \geq 0$ of $\cB$;
\item trace preserving $\ast$-homomorphisms $\pi_s: \cM \rightarrow \mathcal{B}_s$;
\end{enumerate}
satisfying the property:
\[
\mathbb{E}_{\cB_s} \circ \pi_t = \pi_s \circ T_{t-s}, \qquad t \geq s,
\]
where $\mathbb{E}_{\cB_s}: \mathcal{B} \rightarrow \mathcal{B}_s$ are the $\tau_{\cB}$-preserving conditional expectations.
\end{dfn}


\begin{dfn}
We say that the dilation has continuous path if, for every $x\in \cM$ the mapping $\mathbb{R}_{\geq 0} \rightarrow \cB: t\to \pi_t(x)$ is  continuous in measure.
\end{dfn}


 In \cite[Theorem 5.2 (i)]{JungeMei} the following interpolation result was obtained.

\begin{thm}\label{Thm=JungeMeiInterpolation}
Let $(\cM, \tau)$ be a finite von Neumann algebra and let $\cT$ be a Markov semi-group on $\cM$ that admits a  standard Markov dilation. Then the complex interpolation space $[\BigBMO_{\cT}, L_2^\circ(\cM)]_{\frac{2}{p}}$ equals $L_{p}^\circ(\cM)$ with equivalence of norms up to a constant $\simeq p$.
\end{thm}

 \subsection{The Heat semi-group and Euclidean BMO-spaces}\label{Sect=Heat}

In the Euclidean (non-finite) case we describe BMO-spaces separately. For $f \in L_2(\mathbb{R}^n)$ let
 \[
 \widehat{f}(\xi) = (2\pi)^{-\frac{n}{2}}  \int_{\mathbb{R}^n} f(\xi) \exp(i \langle \xi, \eta \rangle) d\xi,
 \]
be its unitary Fourier transform. Define the gradient and Laplace operator
\[
\nabla = \frac{1}{i} \left(\frac{\partial}{\partial x_1}, \ldots, \frac{\partial}{\partial x_n} \right), \qquad
\Delta = - \nabla \cdot \nabla =  \sum_{j=1}^n \frac{\partial^2}{\partial^2 x_j}.
\]
So $\Delta \leq 0$.
For $t \geq 0$ let   $e^{t \Delta}: L_\infty(\mathbb{R}^n) \rightarrow  L_\infty(\mathbb{R}^n)$ be the normal unital completely positive map, which is also described by
\[
\widehat{e^{t \Delta}  f } =  H_t^n \widehat{f}, \qquad  f \in L_\infty(\mathbb{R}^n) \cap L_2(\mathbb{R}^n),
\]
with positive definite function $H_t^n(\xi) =  \exp(- t \Vert \xi \Vert_2^2 ), \xi \in \mathbb{R}^n$, i.e. the Heat kernel. Then $\cS  = (e^{t \Delta})_{t \geq 0}$ is a semi-group of completely positive maps that preserve the Haar integral on $L_\infty(\mathbb{R}^n).$ Moreover, for $f \in  L_\infty(\mathbb{R}^n) \cap L_2(\mathbb{R}^n)$ we have that $e^{t\Delta}(f) \in L_2(\mathbb{R}^n)$ and $t \mapsto e^{t\Delta}(f)$ is continuous for the norm of $L_2(\mathbb{R}^n).$

We may define BMO-spaces with respect to the Heat semi-group as operator spaces as follows. Let $\cM$ be a von Neumann algebra (not necessarily finite) and let $L_\infty(\mathbb{R}^n, \cM) \simeq \cM \otimes L_\infty(\mathbb{R}^n)$ be the space of al $\sigma$-weakly measurable essentially bounded functions $f: \mathbb{R}^n \rightarrow \cM$. We may tensor amplify to get a new Markov semi-group $\cS^{\otimes \cM}  := (\id_{\cM} \otimes e^{t \Delta} )_{t \geq 0}$.
Consider the subspace $L_\infty^\circ(\mathbb{R}^n, \cM)$ of all functions $f \in L_\infty(\mathbb{R}^n, \cM)$ such that $(\id_{\cM} \otimes e^{t \Delta} )(f) \rightarrow 0$ in the $\sigma$-weak topolgy as $t \rightarrow \infty$. On  $L_\infty^\circ(\mathbb{R}^n, \cM)$ we may define a column BMO-norm by,
\[
\Vert f \Vert_{\BMO_{\cS^{\otimes \cM}}^c} = \sup_{t \geq 0} \Vert (\id_{\cM} \otimes e^{t \Delta} )(f^\ast f) -   (\id_{\cM} \otimes e^{t \Delta} )(f)^\ast  (\id_{\cM} \otimes e^{t \Delta} )(f) \Vert^{\frac{1}{2} }.
\]
Then set the row BMO- and the BMO-norm by,
\[\Vert f \Vert_{\BMO_{\cS^{\otimes \cM}}^r} =   \Vert f^\ast \Vert_{\BMO_{\cS^{\otimes \cM}}^c}, \qquad \Vert f \Vert_{\BMO_{\cS^{\otimes \cM}}} = \max(    \Vert f \Vert_{\BMO_{\cS^{\otimes \cM}}^c},   \Vert f\Vert_{\BMO_{\cS^{\otimes \cM}}^r}).\]
The completion of the elements in $L_\infty^\circ(\mathbb{R}^n, \cM) \cap L_2(\mathbb{R}^n, \cM)$ with finite $\Vert f \Vert_{\BMO_{\cS^{\otimes \cM}}}$-norm  with respect to $\Vert f \Vert_{\BMO_{\cS^{\otimes \cM}}}$ is then denoted by $\BMO(\mathbb{R}^n, \cS^{\otimes \cM})$ or simply $\BMO_{\cS^{\otimes \cM}}$.  $\BMO(\mathbb{R}^n, \cS^{\otimes \cM})$  has the operator space structure given by the natural identification,
\[
M_n(\BMO(\mathbb{R}^n, \cS^{\otimes \cM})) =  \BMO(\mathbb{R}^n, \cS^{\otimes M_n(\cM)}).
\]

\subsection{Completely bounded Fourier multipliers}

\begin{dfn}
A symbol $m: \mathbb{R}^n \backslash \{ 0 \} \rightarrow \mathbb{C}$  is called  homogeneous if  for all $\xi \in \mathbb{R}^n \backslash \{ 0 \}$ and $\lambda \in \mathbb{R}_{\geq 0}$ we have $m(\lambda \xi) =  m(\xi)$. For such a symbol we extend it by $m(0) = 0$.
\end{dfn}

  By spectral calculus $m(\nabla)$ is the Fourier multiplier with symbol $m$; more precisely $\widehat{m(\nabla)(f)} = \widehat{m f}$ where we recall that $f \mapsto \widehat{f}$ is the unitary Fourier transform.  The following proposition with just bounds instead of complete bounds is a consequence of the H\"ormander-Mikhlin multiplier theorem. For the complete bounds we base ourselves on  \cite{JMP}. Recall that  $\cS = (e^{t \Delta})_{t \geq 0}$ is the Heat semi-group on $\mathbb{R}^n.$

  We call a function $f \in L_\infty(\mathbb{R}^n,\mathcal{N})$ trigonometric if it is in the linear span of functions $e_{\eta,x}(\xi) = e^{i \langle \eta, \xi \rangle} x , \xi, \eta \in \mathbb{R}^n, x \in \cN.$ Let $\mathcal{A}$ be the $\ast$-algebra of trigonometric functions in $L_\infty(\mathbb{R}^n,\mathcal{N})$.

\begin{prop}\label{Prop=CBMulti}
Let $\cN$ be a semi-finite von Neumann algebra. Let $m:\mathbb{R}^n\backslash \{ 0 \} \rightarrow \mathbb{C}$ be a smooth homogeneous symbol and set $m(0)=0.$ For every trigonometric function $f\in L_\infty(\mathbb{R}^n,\mathcal{N}),$ we have
$$\|(m(\nabla) \otimes {\rm id}_{\cN})(f)\|_{\BMO(L_{\infty}(\mathbb{R}^n) \otimes\mathcal{N},\cS\otimes{\rm id}_{\mathcal{N}})}\leq c_m\|f\|_{\infty},$$
where the constant $c_m$ depends only on the function $m$ (that is, it does not depend either on $\mathcal{N}$ or $f$).
\end{prop}
\begin{proof}
We first note that as  $m(0) = 0$ we find that   $m(\nabla)(f) \in L_\infty^\circ(\mathbb{R}^n, \cN)$ for every $f \in \mathcal{A}$. 	
We check   Conditions (i) and (ii) of \cite[Lemma 2.3]{JMP}.   As $m$ is bounded as a function (by homogeneity) \cite[Remark 2.4]{JMP} immediately gives Condition (i). Next homogenity of $m$ implies that   there exists a constant $c_n$ such that for all multi-indices $\beta$ with $\vert \beta \vert \leq n + 2$,
\[\vert (\partial_\beta  \: m )(\xi) \vert \leq c_n \Vert \xi \Vert_2^{-\vert \beta \vert}, \qquad \xi \in \mathbb{R}^n \backslash \{ 0 \}.\]
This implies    Condition (ii) for the Fourier transform $k = \widehat{m}$ of  \cite[Lemma 2.3]{JMP} by \cite[p. 75, Theorem 6]{SteinInt}. Then \cite[Lemma 2.3]{JMP} shows that  there is a constant $c_m$, only depending on $m$, such that for $f \in \mathcal{A}$   we have
\begin{equation}\label{Eqn=ColumnBd}
\Vert (m(\nabla) \otimes \id_{\cN})(f) \Vert_{\BMO^c( L_\infty( \mathbb{R}^n) \otimes \cN, \cS \otimes \id_{\cN})} \leq  c_m \Vert f \Vert_\infty.
\end{equation}
This yields the column estimate. Further, as we have, for $f \in \mathcal{A}$,
\[
\begin{split}
& \Vert (m(\nabla) \otimes \id_{\cN})(f) \Vert_{\BMO^r(L_\infty(\mathbb{R}^n) \otimes \cN, \cS \otimes \id_{\cN})} \\
& \qquad  = \Vert ( (m^\vee)(\nabla)\otimes \id_{\cN})(\overline{f}) \Vert_{\BMO^c(L_\infty(\mathbb{R}^n) \otimes \cN, \cS \otimes \id_{\cN})},
\end{split}
\]
with $m^\vee(\xi) = \overline{ m(-\xi) }$ we also get the row estimate; in combination with the column estimate \eqref{Eqn=ColumnBd} we see that there is a constant $c_m$ such that for every $f \in \mathcal{A}$   we have
\[
\Vert (m(\nabla) \otimes \id_{\cN}) (f) \Vert_{\BMO(L_\infty(\mathbb{R}^n) \otimes \cN, \cS \otimes \id_{\cN})}  \leq  c_m \Vert f \Vert_\infty.
\]
\end{proof}

\subsection{Double operator integrals}\label{Sect=DOI} We recall the following from \cite{PSW}. Let $\cM$ be a von Neumann algebra (not necessarily finite). Let $A_l \in \cM, 1 \leq l \leq n$ be commuting self-adjoint operators. Briefly set $\boldA = (A_1, \ldots, A_n)$. Let $E: \mathcal{B}(\mathbb{R}^n) \rightarrow \cM$ be the joint spectral measure of $\boldA$ on the Borel sets $\mathcal{B}(\mathbb{R}^n)$. So that we have spectral decompositions $A_l = \int_{\mathbb{R}^n} \xi_l dE(\xi)$ with $\xi_l$ the $l$-th coordinate function.   We define a spectral measure $F: \cB(\mathbb{R}^{2n}) \rightarrow B(L_2(\cM))$ by $F(X \times Y)(x) = E(X) x E(Y)$ where $X,Y \subseteq \mathbb{R}^n$ are Borel sets and $x \in L_2(\cM)$. So $F$ takes values in the projections on $L_2(\cM)$. Then for $\phi: \mathbb{R}^{2n} \rightarrow \mathbb{C}$ a bounded Borel function we set the double operator integral,
\[\cI^{\boldA}_\phi  = \int_{\mathbb{R}^{2n}} \phi(\eta_1, \eta_2) dF(\eta_1, \eta_2) \in B(L_2(\cM)),\]
we shall also use the notation,
\[\cI^{\boldA}_\phi(x) = \int_{\mathbb{R}^{2n}} \phi(\eta_1, \eta_2) dE(\eta_1) x E(\eta_2), \qquad x \in L_2(\cM).\]
In case $\boldA$ is just a single operator $A$ we write $\cI_\phi^A$.

In this paper we shall be interested in extensions of $\cI^{A}_\phi$ to BMO- and $L_p$-spaces associated with $\cM$. Here we record the relation that if $\cM$ is finite and  $A = \sum_{\lambda \in \sigma(A)} \lambda p_\lambda$ has discrete spectrum with $p_\lambda = E(\{ \lambda\})$ then,
 \[\cI^{A}_\phi(x) = \sum_{\lambda, \mu \in \sigma(A)} \phi(\lambda, \mu) p_\lambda x p_\mu, \quad x \in \cM.\]

For $B \in \cM$ self-adjoint set  $\lfloor B \rfloor := \sum_{i \in \mathbb{Z}}   i  \chi_{[i, i+1)}( B)$ with $\chi$ the indicator function. We shall repeatedly make use of the following Lemma \ref{Lem=ConvDOI} without further reference.

 \begin{lem}\label{Lem=ConvDOI}
Let $A \in \cM$ be self-adjoint and for $l \in \mathbb{N}_{\geq 1}$ let $A_l = l^{-1} \lfloor l A \rfloor$. Let $\phi: \mathbb{R}^2 \rightarrow \mathbb{C}$ be continuous. For every $x \in L_2(\cM)$ we have $\Vert \cI_\phi^{A_l}(x) - \cI_\phi^{A}(x)  \Vert_2 \rightarrow 0$ as $l\longrightarrow\infty.$
 \end{lem}
\begin{proof}
We have $\cI_\phi^{A_l} = \cI_{\phi_l}^{A}$ with $\phi_l(\xi) = \phi( l^{-1} \lfloor l \xi \rfloor )$. Then  $\Vert \cI_{\phi_l}^{A}(x) - \cI_{\phi}^{A}(x) \Vert_2 \rightarrow 0$, c.f.  \cite[Lemma 5.1]{CPZ}.
\end{proof}

 \subsection{Vector valued double operator integrals}\label{Sect=DOIvector}
We define vector valued analogues of double operator integrals. To this end suppose that $\cM$ and $\cN$ are  finite von Neumann algebras.  Let $\phi: \mathbb{R}^{n} \times \mathbb{R}^n \rightarrow \cN$ be an essentially bounded $\sigma$-weakly continuous function. Then in particular we also have the same map $\phi: \mathbb{R}^{n} \times \mathbb{R}^n \rightarrow L_2(\cN)$ and this mapping is norm continuous for $L_2(\cN)$. As before let $\boldA$ be an $n$-tuple of mutually commuting self-adjoint operators. For $x \in L_2(\cM)$ we define the double operator integral $\cI_{\phi}^\boldA(x)$ as the unique element in $L_2(\cN) \otimes L_2(\cM)$ that is characterized by,
 \begin{equation}\label{Eqn=DOIvector}
 \langle \cI_{\phi}^\boldA(x), \xi \otimes \eta \rangle = \langle \cI_{ \xi^\ast \circ \phi}^\boldA(x),  \eta  \rangle_{L_2(\cM), L_2(\cM)}, \qquad \xi \in  L_2(\cN), \eta \in L_2(\cM).
 \end{equation}
 Here $\xi^\ast(\eta) = \langle \eta, \xi \rangle$ so that $\xi^\ast \circ \phi: \mathbb{R}^n \times \mathbb{R}^n \rightarrow \mathbb{C}$ is a continuous bounded function and the right hand side of \eqref{Eqn=DOIvector} is the usual double operator integral.

In case the spectrum of $\boldA$ is finite our constructions simplify. We may view $\cI_{\phi}^{\boldA}(x), x \in L_2(\cM) \cap \cM$ as an element of $\cN \otimes \cM$ given by,
\[\cI_{\phi}^{\boldA}(x) = \sum_{i,j \in\sigma(\boldA)}\phi(i,j)\otimes p_ixp_j,\]
where $p_i=\prod_{k=1}^n \chi_{\{i_k\}}(A_k)$ is a spectral projection of $\boldA.$

\subsection{Exterior algebra}
Let $H_{\mathbb{R}}$ be a real Hilbert space and let  $H = H_{\mathbb{R}} \otimes \mathbb{C}$ be its complexified Hilbert space. Let $F^\circ = \mathbb{C} \Omega \oplus \bigoplus_{n=1}^{\infty} H^{\otimes n}$ with unit vector $\Omega$ (the vacuum vector). The pre-inner product on $F^\circ$ is set by,
\[
 \langle \xi_1 \otimes \ldots \otimes \xi_n, \eta_1 \otimes \ldots \otimes \eta_k \rangle = \delta_{n,k} \sum_{\sigma \in S_n}  (-1)^{i(\sigma)} \langle \xi_{\sigma(1)} \otimes \ldots \otimes \xi_{\sigma(n)}, \eta_1 \otimes \ldots \otimes \eta_k \rangle
 \]
where $i(\sigma)$ is the number of inversions on $\sigma$, i.e.. then number of pairs $(a,b)$ with $a < b$ such that $\sigma(b) < \sigma(a)$. Let $F$ be the completion of $F^\circ$ modulo its degenerate part. We denote $\xi_1 \wedge \ldots \wedge \xi_n \in F$ for the equivalence class  of  $\xi_1 \otimes \ldots \otimes \xi_n \in F^\circ$. So with the wedge product $F$ is the usual exterior algebra (or Clifford algebra with the zero quadratic form; note if $\dim(H) < \infty$ also $\dim(F) < \infty$).
For $\xi \in H$ we set
\[
l(\xi) \eta = \xi \wedge \eta, \qquad l^\ast(\xi) (\eta_1 \wedge \ldots \wedge \eta_n)  = \sum_{k=1}^n (-1)^{k} \: \langle \eta_k, \xi \rangle \eta_1 \wedge \ldots \wedge \widehat{\eta_k} \wedge \ldots \wedge \eta_n,
\]
and extend them to bounded operators on $F$. Here $\widehat{\eta_k}$ means that the $k$-th wedge term is excluded from the term. $l^\ast(\xi)$ is the adjoint of $l(\xi)$. We set $s(\xi) = l(\xi) + l^{\ast}(\xi)$. And further $\Gamma := \Gamma(H_{\mathbb{R}}) := \{ s(\xi) \mid \xi \in H_{\mathbb{R}}  \}''.$
  We record the fundamental property of the exterior algebra:
\begin{equation}\label{anticommutation}
s(\xi)s(\eta)+s(\eta)s(\xi)=2\langle \xi,\eta\rangle,\quad \xi,\eta\in H_{\mathbb{R}}.
\end{equation}
The von Neumann algebra $\Gamma$ has faithful normal tracial state $\tau_\Omega(x) = \langle x \Omega, \Omega \rangle$, the vacuum state, c.f. \cite{BozejkoSpeicher} for these results in  greater generality.

\section{Discrete Markov dilations}\label{Sect=Disc}
We show how Markov dilations of discrete semi-groups can be used to get Markov dilations of a continuous one through ultraproduct techniques. In particular we show that we can always guarantee path continuity (in measure topology) of Markov dilations for finite von Neumann algebras.

In the special case, when $\sfG=\mathbb{R}_{\ge 0}$, the definition below coincides with Definition 2.3 above.

\begin{dfn} Let $\sfG$ be a subsemi-group of $\mathbb{R}_{\geq 0}$. We say that a Markov semi-group $(T_t)_{t\in\sfG}$ acting on the probability space $(\cM,\tau)$ admits a standard Markov dilation if there exist:
\begin{enumerate}[{\rm (i)}]
\item a finite von Neumann algebra $(\cB, \tau_{\cB})$;
\item an increasing filtration $(\cB_t)_{t\in\sfG}$ of $\cB$;
\item trace preserving $\ast$-homomorphisms $\pi_t:\cM \rightarrow \cB_t,$ $t\in\sfG;$
\end{enumerate}
such that for $t,s \in\sfG$ with $s \geq t$, we have
\[
\mathbb{E}_{\cB_t} \circ \pi_s=\pi_t\circ T_{s-t}.
\]
Here, $\mathbb{E}_{\cB_t}:  \cB \rightarrow \cB_t$  are the $\tau_{\cB}$-preserving conditional expectations.
\end{dfn}

\begin{thm}\label{discretization theorem}  Let $(\mathcal{M},\tau)$ be a finite von Neumann algebra.  Let $\mathcal{T}=(T_t)_{t\geq0}$ be a Markov semi-group. If the Markov semi-group $(T_t)_{t\in\epsilon\mathbb{N}_{\geq 0}}$ admits a standard Markov dilation for every $\epsilon>0,$ then so does $\mathcal{T}.$ Moreover, the dilation has continuous path.
\end{thm}

We prove this theorem in the next couple of lemmas. We shall repeatedly make use  of the fact that the measure topology and the $L_2$-topology coincide on the unit ball of a finite von Neumann algebra.

\begin{lem}\label{binary lemma} Set the semi-group $\sfG = \cup_{l \in \mathbb{N}_{\geq 0}} 2^{-l} \mathbb{N}_{\geq 0}$. Under the assumptions of Theorem \ref{discretization theorem}, $(T_t)_{t\in\sfG}$ admits a standard Markov dilation.
\end{lem}
\begin{proof} For $l\geq 0,$ let $\sfG_l =2^{-l}\mathbb{N}_{\geq 0}$ so that  $\sfG = \cup_{l\geq0}\sfG_l.$ We see $\sfG$ as a subsemi-group of $\mathbb{R}_{\geq 0}$ and equip it with the Euclidean topology.
By assumption (with $\epsilon=2^{-l}$), there exists:
\begin{enumerate}[{\rm (i)}]
\item a finite von Neumann algebra $(\cB^l, \tau_{\cB^l});$
\item an increasing filtration  $(\cB_m^l)_{m \in \sfG_l}$ of $\cB^l;$
\item trace preserving $\ast$-homomorphisms $\pi_m^l: \cM \rightarrow \cB_m^l;$
\end{enumerate}
such that for $m,k \in\sfG_l$ with $k \geq m,$ we have
$$\mathbb{E}_{\cB_m^l}\circ \pi_k^l=\pi_m^l \circ T_{k-m}.$$
Set Ocneanu ultrapowers (see e.g. \cite{AndoHaagerup}) $(\cB,\tau_{\cB})=\prod_{l,\omega}(\cB^l, \tau_{\cB^l})$ and $\cB_m = \prod_{l, \omega}(\cB^l_m, \tau_{\cB_l})$ for $m\in \sfG.$ The second ultraproduct runs over large enough $l,$ namely such that $m \in \sfG_l.$

Fix $m_1,m_2\in \sfG$ such that $m_2 \leq m_1$ and choose $l_0$ such that $m_1,m_2\in\sfG_l$ for all $l\geq l_0.$ We have
$$\cB_{m_1}=\prod_{l, \omega}(\cB^l_{m_1},\tau_{\cB_l}),\quad \cB_{m_2}=\prod_{l, \omega}(\cB^l_{m_2}, \tau_{\cB_l}),$$
where ultrafilter runs over all $l\geq l_0.$ Since for every $l\geq l_0$ we have
$$\cB^l_{m_2}\subset \cB^l_{m_1},$$
it follows that
$$\cB_{m_2}\subset \cB_{m_1}.$$
Therefore, we have an increasing filtration.
Let $\mathbb{E}_{\cB_m}$ be the trace preserving  conditional expectation of $\cB$ onto $\cB_m.$ Note that
$$\mathbb{E}_{\cB_m}((x_l)_{l,\omega})=(\mathbb{E}_{\cB_m^l}(x_l))_{l,\omega}.$$

For $m \in \sfG,$ define a trace preserving $*-$homomorphism $\pi_m:\cM\to\cB_m$ by the formula
$$\pi_m(x)=(\pi_m^l(x))_{l, \omega}.$$
For $m,k\in\sfG$ with $k \geq m$, we have
$$(\mathbb{E}_{\cB_m})(\pi_k(x))=(\mathbb{E}_{\cB_m^l}(\pi_k^l(x)))_{l,\omega}=(\pi_m^l(T_{k-m}(x)))_{l,\omega}=\pi_m(T_{k-m}(x)).$$
\end{proof}

\begin{lem}\label{first l2 estimate} Let $\sfG$ be a subsemi-group in $\mathbb{R}_+$ and let $(T_t)_{t\geq0}$ be a Markov semi-group. If $(T_t)_{t\in\sfG}$ admits a standard Markov dilation, then for every $x\in\mathcal{M},$
$$\|\pi_t(x)-\pi_s(x)\|_2^2\leq 2\|x\|_2\|x-T_{|s-t|}(x)\|_2,\quad t,s\in\sfG.$$
\end{lem}
\begin{proof} Without loss of generality, $s\geq t.$ For $x\in\cM,$ we have by \cite[p. 211 (3) and (4)]{TakII},
\[
\begin{split}
&   \tau_{\cB}(\pi_t(x)^{\ast}\pi_s(x))=
(\tau_{\cB}\circ \mathbb{E}_{\cB_t})(\pi_t(x)^{\ast}\pi_s(x)) \\
= & \tau_{\cB}(\pi_t(x)^{\ast}\mathbb{E}_{\cB_t}(\pi_s(x)))=\tau_{\cB}(\pi_t(x)^{\ast}\pi_t(T_{s-t}(x))).
\end{split}
\]
Since $\pi_t$ is trace preserving, it follows that
$$\tau_{\cB}(\pi_t(x)^{\ast}\pi_s(x))=\tau(x^{\ast}T_{s-t}(x)),\quad t,s\in\sfG,\quad s\geq t.$$
Similarly,
$$\tau_{\cB}(\pi_s(x)^{\ast}\pi_t(x))=\tau(xT_{s-t}(x^{\ast})),\quad t,s\in\sfG,\quad s\geq t.$$
Therefore, we have
\[
\begin{split}
\|\pi_t(x)-\pi_s(x)\|_2^2= & \tau_{\cB}(\pi_t(x)^{\ast}\pi_t(x))+\tau_{\cB}(\pi_s(x)^{\ast}\pi_s(x)) \\
& - \tau_{\cB}(\pi_t(x)^{\ast}\pi_s(x))-\tau_{\cB}(\pi_s(x)^{\ast}\pi_t(x)) \\
=&
2\tau(x^{\ast}x)-\tau(x^{\ast}T_{s-t}(x))-\tau(xT_{s-t}(x^{\ast})) \\
= &\tau(x^{\ast} (x-T_{s-t}(x)))+\tau(x (x-T_{s-t}(x))^{\ast}).
\end{split}
\]
Applying the Cauchy-Schwarz inequality, we conclude the argument.
\end{proof}

%

We call a family $(\pi_t)_{t\geq0}$ of $\ast$-homomorphisms $\cM \rightarrow \cB$ continuous in the point-measure topology if for every $x \in \cM$ we have that  $t \mapsto \pi_t(x)$ is   continuous in measure.

\begin{lem}\label{pi continuity lemma} Let $\sfG$ be a dense subsemi-group in $\mathbb{R}_+$ and let $(T_t)_{t\geq0}$ be a Markov semi-group. If $(T_t)_{t\in\sfG}$ admits a standard Markov dilation, then $(\pi_t)_{t\in\sfG}$ extends to a family $(\pi_t)_{t\geq0}$  of trace preserving $^{\ast}$-homomorphisms so that, for every $x\in\cM,$ the mapping $t\to\pi_t(x)$ is continuous in measure.
\end{lem}
\begin{proof}
From the fact that Markov semi-groups are by definition continuous in measure this is a direct consequence of Lemma \ref{first l2 estimate}.
\end{proof}

\begin{proof}[Proof of Theorem \ref{discretization theorem}] Let $\sfG = \cup_{l \in \mathbb{N}_{\geq 0}} 2^{-l} \mathbb{N}_{\geq 0}$ be the set of all non-negative binary rationals. By Lemma \ref{binary lemma}, $(T_t)_{t\in\sfG}$ admits a standard Markov dilation. Set
$$\cB_t=\Big(\bigcup_{\substack{u \leq t\\ u\in\sfG}}\cB_u\Big)''.$$
In the following equations we shall take the limit $u \rightarrow t$ over the sets in the subscript of the limit. By construction, we have
$$\mathbb{E}_{\cB_t}(w)=\lim_{\substack{u\leq t\\ u\in\sfG }}\mathbb{E}_{\cB_u}(w),\quad w\in\cB,$$
in the $L_2$-topology. Let $(\pi_t)_{t\geq0}$ be an $L_2$-continuous family of trace preserving $*$-homomorphisms constructed in Lemma \ref{pi continuity lemma}.

If $s\geq t$ and $s\in\sfG,$ then
$$\mathbb{E}_{\cB_t}(\pi_s(x))=\lim_{\substack{u\leq t\\ u\in\sfG}}\mathbb{E}_{\cB_u}(\pi_s(x))=\lim_{\substack{u\leq t\\ u\in\sfG}}\pi_u(T_{s-u}(x)).$$
By assumption, we have
$$\lim_{\substack{u\leq t\\ u\in\sfG}}T_{s-u}(x)=\lim_{\substack{u\leq t\\ u\in\sfG}}T_{t-u}(T_{s-t}(x))=T_{s-t}(x)$$
 in the $L_2-$norm. Each $\pi_u,$ $u\geq0,$ is a trace preserving $*$-homomorphism and  therefore  contracts the $L_2-$norm. Hence, as $(\pi_t)_{t \in \sfG}$ are $\ast$-homomorphisms of a  Markov dilation we see   by Lemma \ref{pi continuity lemma} that we have a limit in measure,
$$\mathbb{E}_{\cB_t}(\pi_s(x))=\lim_{\substack{u\leq t\\ u\in\sfG}}\pi_u(T_{s-t}(x)) \stackrel{L.\ref{pi continuity lemma}}{=}\pi_t(T_{s-t}(x)).$$

Let now $s\geq t\geq0.$ If $s_k\in\sfG,$ $s_k\searrow s,$ then $\pi_{s_k}(x)\to\pi_s(x)$  in the $L_2-$norm. Since $\mathbb{E}_{\cB_t}$ contracts the $L_2-$norm, it follows that
$$\mathbb{E}_{\cB_t}(\pi_s(x))=\lim_{k\to\infty}\mathbb{E}_{\cB_t}(\pi_{s_k}(x))=\lim_{k\to\infty}\pi_t(T_{s_k-t}(x))$$
We have $T_{s_k-t}(x)=T_{s_k-s}(T_{s-t}(x)).$ By assumption, $T_{s_k-s}(x)\to x$   in the $L_2-$norm and  hence,
$$\mathbb{E}_{\cB_t}(\pi_s(x))=\lim_{k\to\infty}\pi_t(T_{s_k-t}(x))= \pi_t(T_{s-t}(x)).$$
This completes the proof.
\end{proof}

As a direct corollary of Lemma \ref{first l2 estimate} we obtain the following automatic continuity property for finite von Neumann algebras.

\begin{cor}   Let $(\mathcal{M},\tau)$ be a finite von Neumann algebra. Let $\cT = (T_t)_{t\geq 0}$ be a Markov semi-group that admits a standard Markov dilation. Then the dilation has continuous path.
\end{cor}

\section{Markov dilations for semi-groups of double operator integrals}\label{Sect=DOIDilate}
In \cite{RicardDilation} Ricard proved that discrete semi-groups of Schur multipliers admit a standard Markov dilation. In this section we show that also semi-groups of double operator integrals have a standard Markov dilation.  In what follows, this fact, together with Theorem \ref{Thm=JungeMeiInterpolation} allows us to interpolate between respective ${\rm BMO}-$space and $L_2-$space. 

\begin{thm}\label{Thm=DOIMarkovDilation} Let $(\cM,\tau)$ be a finite von Neumann algebra and let $\cT = (T_t)_{t\geq 0}$ be a Markov semi-group such that there exists  $A=A^*\in\mathcal{M}$ and $\phi_t: \mathbb{R}^2 \rightarrow \mathbb{C}$ continuous  with $T_t = \cI_{\phi_t}^A, t \geq 0$. Then  the semi-group $\cT$ admits a standard Markov dilation.
\end{thm}

The crucial part of the argument is similar to that of Ricard \cite{RicardDilation}.

\begin{prop}\label{Thm=RicardDilation}
Let $(\cM, \tau)$ be a finite von Neumann algebra and let $A=A^*\in\mathcal{M}$ be such that ${\rm spec}(A)\subset\mathbb{Z}.$ Let $\phi:\mathbb{Z}^2 \rightarrow \mathbb{R}$ be a positive matrix such that for all $i$ we have $\phi(i,i) =1.$ The semi-group $((\cI_{\phi}^A)^n)_{n\in\mathbb{N}_{\geq0}}$ acting on $\cM$ admits a standard Markov dilation.
\end{prop}
\begin{proof} As $A$ is bounded, we have ${\rm spec}(A)\subset\{-n,1-n,\cdots,n-1,n\}$ for some $n\in\mathbb{N}.$ Denote for brevity $p_i=\chi_{\{i\}}(A),$ $-n\leq i\leq n.$ As $\phi$ is positive, the expression
\begin{equation}\label{Eqn=PhiInner}
\langle \xi, \eta \rangle = \sum_{i,j =-n}^n\phi(i,j)\xi_i\eta_j.
\end{equation}
defines a positive  (possibly degenerate) inner product on $\mathbb{R}^{2n+1}.$  Let $H_\mathbb{R}$  be $\mathbb{R}^{2n+1}$ {\it equipped with inner product} \eqref{Eqn=PhiInner} and quotienting out the degenerate part. Construct the associated exterior algebra $\Gamma = \Gamma(H_{\mathbb{R}})$ from it.

Let $\{e_i\}_{i=-n}^n$ be the standard orthonormal basis of $\mathbb{R}^{2n+1}$ viewed as elements (e.g. equivalence classes) of $H_{\mathbb{R}}.$
Let $\cB = \cM \otimes  \Gamma^{\otimes \infty}$ with tensor product trace $\tau_{\cB} = \tau \otimes \tau_{\Omega}^{\otimes \infty}$ (tensor products constructed from the vacuum state, see \cite{TakIII} for infinite tensor powers). We infer from \eqref{anticommutation} that $s(e_i)^2=1.$ Define a unitary
\[
u = \sum_{i=-n}^n p_i \otimes s(e_i)  \in \cM \otimes \Gamma,
\]
which we view as a unitary in the first two tensor legs of $\cB = \cM \otimes \Gamma^{\otimes \infty} = \cM \otimes \Gamma \otimes \Gamma^{\otimes \infty}$ by identifying it with $u \otimes 1_{\Gamma}^{\otimes \infty}$.
Let $S$ be the tensor shift on $\Gamma^{\otimes \infty}$ determined by,
$$S (x_1 \otimes \ldots \otimes x_m \otimes 1 \otimes \ldots) = 1 \otimes x_1 \otimes \ldots \otimes x_m \otimes 1 \otimes \ldots,$$
and then set the $\ast$-homomorphism $\beta(x) =   u^\ast (\id_{\cM} \otimes S)(x) u.$
For $k\geq0,$ define a trace-preserving $\ast$-homomorphism  $\pi_k:\cM\to\cB$ as follows:
$$\pi_0: x \mapsto x \otimes 1 \otimes 1 \ldots$$
and
$$\pi_k: x \mapsto (\beta^k \circ \pi_0)(x), \qquad k \in \mathbb{N}_{\geq 0}.$$
Using induction we obtain that,
\begin{equation}\label{Eqn=Explicit}
\pi_k(x) = \sum_{i,j=-n}^np_i x p_j  \otimes (s(e_i) s(e_j))^{\otimes k }  \otimes 1^{\otimes \infty}_\Gamma \in \cB.
\end{equation}
Indeed,
\[
\begin{split}
\pi_{k+1}(x) = &  u^\ast  (\id_\cM \otimes S)  \left(  \sum_{i,j=-n}^np_i x p_j  \otimes (s(e_i) s(e_j))^{\otimes k }  \otimes 1^{\otimes \infty}_\Gamma  \right) u \\
= & u^\ast   \left(  \sum_{i,j=-n}^np_i x p_j  \otimes 1_{\Gamma} \otimes  (s(e_i) s(e_j))^{\otimes k }  \otimes 1^{\otimes \infty}_\Gamma  \right) u \\
= &  \sum_{i,j=-n}^np_i x p_j  \otimes  (s(e_i) s(e_j))^{\otimes k+1 }  \otimes 1^{\otimes \infty}_\Gamma.
\end{split}
\]

Define the increasing family of subalgebras $\cB_m$ as the von Neumann algebras $\mathcal{M} \otimes \Gamma^{\otimes m} \otimes 1_\Gamma^{\infty} \subseteq \cB$.
If $k \geq m$ and if $x\in\cM$ is such that $p_ixp_j=x,$ then a direct computation yields
$$\mathbb{E}_m(x\otimes (s(e_i) s(e_j))^{\otimes k}  \otimes 1^{\otimes \infty}_\Gamma)=\tau_{\Omega}(s(e_i) s(e_j))^{k-m} x\otimes (s(e_i) s(e_j))^{\otimes m}  \otimes 1^{\otimes \infty}_\Gamma.$$
We get that for $x\in\cM$ and for $k \geq m$,
$$(\mathbb{E}_m \circ \pi_k)(x)=\sum_{i,j =-n}^n \tau_{\Omega}(s(e_i) s(e_j))^{k-m} \:  p_i x p_j \otimes (s(e_i)s(e_j) )^{\otimes m} \otimes 1^\infty_\Gamma.$$
By \eqref{anticommutation} and \eqref{Eqn=PhiInner},
$$\tau_{\Omega}(s(e_i) s(e_j)) = \langle e_i, e_j \rangle = \phi(i,j).$$
Therefore,
$$\tau_{\Omega}(s(e_i) s(e_j))^{k-m} \:  p_i x p_j = \phi(i,j)^{k-m} p_i x p_j =(\cI_\phi^A)^{k-m}(p_i x p_j) =p_i((\cI_\phi^A)^{k-m}(x))p_j.$$
Hence,
\begin{equation}\label{Eqn=PrevLine}
(\mathbb{E}_m \circ \pi_k)(x)=\sum_{i,j =-n}^n p_i((\cI_\phi^A)^{k-m}(x))p_j \otimes (s(e_i)s(e_j) )^{\otimes m} \otimes 1^\infty_\Gamma.
\end{equation}
By \eqref{Eqn=Explicit} and   \eqref{Eqn=PrevLine} we get,
$$(\mathbb{E}_m \circ \pi_k)(x) = \pi_m((\cI_\phi^A)^{k-m}(x)).$$
This completes the proof.
\end{proof}

The passage to operators $A$ with arbitrary spectrum (not just integral) requires the approximation result  below.

\begin{prop}\label{markov approximation prop} Let $T$ and $T_l$ be unital trace preserving maps on $\cM$ such that $T_l(x)\to T(x)$ strongly for all $x\in\mathcal{M}.$ If every semi-group $(T_l^n)_{n\in\mathbb{N}_{\geq0}}$ admits a standard Markov dilation, then so does the semi-group $(T^n)_{n\in\mathbb{N}_{\geq0}}.$
\end{prop}
\begin{proof} By assumption, there exists
\begin{enumerate}[{\rm (i)}]
\item a finite von Neumann algebra $(\cB^l, \tau_{\cB^l});$
\item an increasing filtration $(\cB_m^l)_{m\in \mathbb{N}_{\geq 0}};$
\item trace preserving $\ast$-homomorphisms $\pi_m^l: \cM \rightarrow \cB_m^l;$
\end{enumerate}
such that for $m,k\in \mathbb{N}_{\geq 0}$ with $k\geq m,$ we have
\[
\mathbb{E}_{\cB_m^l}\circ \pi_k^l = \pi_m^l \circ T_l^{k-m}.
\]
Fix a non-principal ultrafilter $\omega$ on $\mathbb{N}_{\geq 1}$ and let $\cB = \prod_{l, \omega} (\cB^l, \tau_l)$ and $\cB_m = \prod_\omega (\cB^l_m, \tau_l)$ be the Ocneanu ultrapowers, see \cite{AndoHaagerup}.
Let $\mathbb{E}_m: \cB \rightarrow \cB_m$ be the expectation preserving the ultraproduct trace $\tau_{\cB}$ on $\cB$. We have that $\{\cB_m\}_{m\geq0}$ is an increasing filtration of subalgebras of $\cB$  and that
\[
\mathbb{E}_m((x_l)_{l, \omega}) =(\mathbb{E}_m^l(x_l))_{l, \omega}.
\]
For every $m\geq0,$ define a trace-preserving $\ast$-homomorphism $\pi_m:\cM\to\cB_m$ by setting
$$\pi_m: \cM \rightarrow \cB_m: x \mapsto (\pi_m^l(x))_{l, \omega}.$$
We find that for $x \in \cM$, $k,m \in \mathbb{N}_{\geq 0}$ and $k \geq m$,
\begin{equation}\label{Eqn=MarkovEq}
(\mathbb{E}_m \circ \pi_k)(x) = \Big((\mathbb{E}_m^l \circ \pi_k^l)(x)\Big)_{l, \omega}=\Big((\pi_m^l\circ T_l^{k-m})(x)\Big)_{l, \omega}.
\end{equation}
Since $\pi_m^l$ is trace preserving, it follows that
\[
\Big\Vert \Big((\pi_m^l\circ T_l^{k-m})(x)\Big)_{l, \omega}-\Big((\pi_m^l\circ T^{k-m})(x)\Big)_{l, \omega}\Big\Vert_2=\lim_{l\to\omega}\Big\Vert T_l^{k-m}(x)-T^{k-m}(x)\Big\Vert_2.
\]
By the triangle inequality, we have
\[
\begin{split}
& \Big\Vert T_l^{k-m}(x)-T^{k-m}(x)\Big\Vert_2\leq\sum_{j=0}^{k-m-1}\Big\Vert\Big(T_l^{k-m-j-1}\circ (T_l-T)\circ T^j\Big)(x)\Big\Vert_2\\
\leq & \sum_{j=0}^{k-m-1}\|T_l\|_{L_2\to L_2}^{k-m-1-j}\Big\Vert (T_l-T)(T^jx)\Big\Vert_2\leq\sum_{j=0}^{k-m-1}\Big\Vert (T_l-T)(T^jx)\Big\Vert_2.
\end{split}
\]
Therefore,
\[
\begin{split}
&\Big\Vert \Big((\pi_m^l\circ T_l^{k-m})(x)\Big)_{l, \omega}-\Big((\pi_m^l\circ T^{k-m})(x)\Big)_{l, \omega}\Big\Vert_2\\
\leq & (k-m)\max_{0\leq j<k-m}\limsup_{l\to\infty}\Big\Vert (T_l-T)(T^jx)\Big\Vert_2.
\end{split}
\]
By assumption, we have that
$$(T_l-T)(T^jx)\to0,\quad l\to\infty,$$
 in $L_2$-norm.  We conclude that
$$\Big((\pi_m^l\circ T_l^{k-m})(x)\Big)_{l, \omega}-\Big((\pi_m^l\circ T^{k-m})(x)\Big)_{l, \omega}=0.$$
Substituting into \eqref{Eqn=MarkovEq}, we obtain
$$(\mathbb{E}_m \circ \pi_k)(x) = \Big((\pi_m^l\circ T^{k-m})(x)\Big)_{l, \omega}=\pi_m(T^{k-m}(x)).$$
\end{proof}

\begin{proof}[Proof of Theorem \ref{Thm=DOIMarkovDilation}] Fix $l\in\mathbb{N}_{\geq1}$ and $\varepsilon > 0$.  The matrix $(\phi_{\epsilon}(\frac{i}{l},\frac{j}{l}))_{i,j\in\mathbb{Z}}$ and the operator
	\[
	\lfloor lA\rfloor := \sum_{i \in \mathbb{Z}}   i  \chi_{[i, i+1)}( l A)
	\]
satisfy the condition of Proposition \ref{Thm=RicardDilation}. Hence, the semi-group $((\cI_{\phi_{\epsilon}}^{\frac1l\lfloor lA\rfloor})^n)_{n\in\mathbb{N}_{\geq0}}$ admits a standard Markov dilation. Since $\phi_{\epsilon}$ is continuous, it follows that
$$\cI_{\phi_{\epsilon}}^{\frac1l\lfloor lA\rfloor}(x)\to \cI_{\phi_{\epsilon}}^A(x),\quad l\to\infty,$$
in $L_2$. By Proposition \ref{markov approximation prop}, the semi-group $((\cI_{\phi_{\epsilon}}^A)^n)_{n\in\mathbb{N}_{\geq0}}$ admits a standard Markov dilation. By Theorem \ref{discretization theorem}  so does the semi-group $(\cI_{\phi_t}^A)_{t\geq0}.$
\end{proof}

\section{Complements on reversed Markov dilations}\label{Sect=ReversedDilateFix}

As we believe these results are of independent use, we also state the corresponding results for reversed Markov dilations. These shall not be used in the current paper. The proofs are completely analogous to the proofs in Sections \ref{Sect=Disc} and \ref{Sect=DOIDilate}.

\begin{dfn}
Let $\mathcal{T} = (T_t)_{t \geq 0}$ be a Markov semi-group on a finite von Neumann algebra $(\cM, \tau)$. Let $\sfG$ be a subsemi-group of $\mathbb{R}_{\geq 0}$. We say that $(T_t)_{t \in \sfG}$ admits  a  reversed Markov dilation if there exist:
\begin{enumerate} [{\rm (i)}]
\item a   finite von Neumann algebra $(\mathcal{B}, \tau_{\cB})$;
\item a decreasing filtration $\mathcal{B}_s, s \geq 0$ with conditional expectations $\mathbb{E}_{\cB_s}: \mathcal{B} \rightarrow \mathcal{B}_s$;
\item  trace preserving $\ast$-homomorphisms $\pi_s: \cM \rightarrow \mathcal{B}_s$
\end{enumerate}
such that the following property holds
\[\mathbb{E}_{\cB_t} \circ \pi_s = \pi_t \circ T_{t-s}, \qquad s,t \in \sfG, t \geq s.\]
\end{dfn}

\begin{dfn}
We say that a reversed Markov dilation has continuous path if for every $x \in \cM$ the mapping $\mathbb{R}_{\geq 0} \rightarrow \cB: t\to \pi_t(x)$  is  continuous in measure.
\end{dfn}
	
In the same way as we proved Theorem \ref{discretization theorem}  we may now obtain the following result.

\begin{thm}  Let $\mathcal{T}=(T_t)_{t\geq0}$ be a Markov semi-group. If $(T_t)_{t\in\epsilon\mathbb{N}_{\geq 0}}$ admits a reversed Markov dilation for every $\epsilon>0,$ then so does $\mathcal{T}.$ Moreover, the dilation has continuous path.
\end{thm}

In \cite[p. 4370]{RicardDilation} Ricard shows that a semi-group of Schur multipliers $(T_\phi^k)_{k \in \mathbb{N}_{\geq 0}}$ admits a reversed Markov dilation. By essentially the same argument as in Theorem \ref{Thm=DOIMarkovDilation} we also get reversed Markov dilations for double operator integrals.

\begin{thm}\label{Thm=DOIStandardMarkovDilationFix}  Let $(\cM,\tau)$ be a finite von Neumann algebra and let $A=A^*\in\mathcal{M}.$ Let $(\cI_{\phi_t}^A)_{t\geq0}$ be a Markov semi-group of double operator integrals. If each $\phi_t$ is continuous, then this semi-group admits a reversed Markov dilation.
\end{thm}

\section{Transference of multipliers and BMO-spaces}\label{Sect=ComBMO}
Fix a Lipschitz function $f: \mathbb{R} \rightarrow \mathbb{R}$ and assume $\Vert f' \Vert_\infty \leq 1$.
For $f: \mathbb{R} \rightarrow \mathbb{R}$ we set the divided difference function $f^{[1]}: \mathbb{R}^2 \rightarrow \mathbb{R}$ by
\[
f^{[1]}(\lambda, \mu) =
\left\{
\begin{array}{ll}
\frac{f(\lambda) - f(\mu)}{ \lambda - \mu}, & \lambda \not = \mu, \\
0, & \lambda = \mu.
\end{array}
\right.
\]
The main result we prove in this section is the following theorem.

\begin{thm}\label{bmo boundedness} Let $(\cM, \tau)$ be a finite von Neumann algebra and let $f:\mathbb{R}\to\mathbb{R}$ be Lipschitz with $\Vert f ' \Vert_\infty \leq 1$. For every $A=A^*\in\mathcal{M},$
\begin{enumerate}[{\rm (i)}]
\item\label{bmoi} the semi-group $\mathcal{I}^A = ( \cI_{e^{-tF}}^A )_{t\geq0}$ with
$$F(\lambda,\mu)=|\lambda-\mu|^2+|f(\lambda)-f(\mu)|^2,\quad \lambda,\mu\in\mathbb{R},$$
is Markov;
\item\label{bmoii} the operator $\cI_{f^{[1]}}^A$ maps $\mathcal{M}$ to $\BMO(\cM,\cI^A)$ and its norm is bounded by an absolute constant $c_{abs}$.
\end{enumerate}
\end{thm}

For $\eta\in \mathbb{R}^2,$ let $e_{\eta}\in L_\infty(\mathbb{R}^2)$ be defined as
\[
e_{\eta}(\xi) = e^{i \langle \xi, \eta\rangle}.
\]
For $A=A^*\in \cM$ with finite spectrum, define a unitary element $U^A\in L_{\infty}(\mathbb{R}^2)\otimes\cM$ by setting
$$U^A=\int_{\mathbb{R}}e_{(\lambda,f(\lambda))}\otimes dE_A(\lambda).$$
where $\{E_A(\lambda)\}_{\lambda\in\mathbb{R}}$ is a spectral family of $A.$ Due to the finiteness assumption on the spectrum, the convergence of the integral follows automatically (in fact, integral is a finite sum of operators).

Define the $*$-monomorphism $\varphi_A:\mathcal{M}\to L_{\infty}(\mathbb{R}^2)\otimes\cM$ by setting
\[\varphi_A(x)=U^A(1\otimes x)(U^A)^{\ast},\quad x\in\cM.\]
Let $m_0$ be a smooth homogeneous multiplier such that
\[m_0(\xi_1,\xi_2)=\frac{\xi_2}{\xi_1}\mbox{ when }|\xi_2|\leq|\xi_1|.\]

Both statements of Theorem \ref{bmo boundedness} are proved through the following transference lemma.

\begin{lem}\label{transference lemma} If, in the setting of Theorem \ref{bmo boundedness}, $A$ has finite spectrum, then
\[\varphi_A\circ\cI_{e^{-tF}}^A=(e^{t \Delta} \otimes{\rm id}_{\mathcal{M}})\circ \varphi_A,\quad \varphi_A\circ\cI_{f^{[1]}}^A=(m_0(\nabla)\otimes {\rm id}_{\mathcal{M}})\circ\varphi_A.\]
\end{lem}
\begin{proof} By definition of $\varphi_A,$ we have
$$\varphi_A(x)=\iint_{\mathbb{R}^2}e_{(\lambda-\mu,f(\lambda)-f(\mu))}\otimes dE_A(\lambda)xdE_A(\mu).$$
Clearly,
\[
\begin{split}
e^{t\Delta}(e_{(\lambda-\mu,f(\lambda)-f(\mu))})=&     e^{-tF(\lambda,\mu)}e_{(\lambda-\mu,f(\lambda)-f(\mu))}\\
(m_0(\nabla))(e_{(\lambda-\mu,f(\lambda)-f(\mu))})=&f^{[1]}(\lambda,\mu)e_{(\lambda-\mu,f(\lambda)-f(\mu))}.
\end{split}
\]
Therefore,
\[\begin{split}
& (e^{t \Delta} \otimes{\rm id}_{\mathcal{M}})(\varphi_A(x))=\iint_{\mathbb{R}^2}e^{-tF(\lambda,\mu)}e_{(\lambda-\mu,f(\lambda)-f(\mu))}\otimes dE_A(\lambda)xdE_A(\mu) \\
= & \iint_{\mathbb{R}^2}e_{(\lambda-\mu,f(\lambda)-f(\mu))}\otimes dE_A(\lambda)(\cI_{e^{-tF}}^A(x))dE_A(\mu)=\varphi_A(\cI_{e^{-tF}}^A(x))
\end{split}
\]
and
\[
\begin{split}
& (m_0(\nabla)\otimes{\rm id}_{\mathcal{M}})(\varphi_A(x))=\iint_{\mathbb{R}^2}f^{[1]}(\lambda,\mu)e_{(\lambda-\mu,f(\lambda)-f(\mu))}\otimes dE_A(\lambda)xdE_A(\mu) \\
= &\iint_{\mathbb{R}^2}e_{(\lambda-\mu,f(\lambda)-f(\mu))}\otimes dE_A(\lambda)(\cI_{f^{[1]}}^A(x))dE_A(\mu)=\varphi_A(\cI_{f^{[1]}}^A(x)).
\end{split}
\]
\end{proof}

Next we prove each of the statements \ref{bmo boundedness} \eqref{bmoi} and \eqref{bmoii} in the following subsections.

\subsection*{Proof of Theorem \ref{bmo boundedness} \eqref{bmoi}}

\begin{lem}\label{approximate markov} Let $(\cM, \tau)$ be a finite von Neumann algebra. Let $G:\mathbb{R}^2\to\mathbb{R}$ be a continuous function. If $( \cI_{e^{-tG}}^A )_{t\geq0}$ is Markov for every $A=A^*\in\cM$ with finite spectrum, then $( \cI_{e^{-tG}}^A )_{t\geq0}$ is Markov for every $A=A^*\in\cM.$
\end{lem}
\begin{proof} Let $A=A^*\in\mathcal{M}$ and let $A_l=l^{-1}\lfloor lA\rfloor$ for $l\geq 1.$ If $\mathcal{N}$ is a finite von Neumann algebra and if $ x\in\mathcal{N} \otimes\mathcal{M}$ is such that $0\leq x\leq 1,$ then
$$0\leq ({\rm id}_{\mathcal{N}}\otimes\cI_{e^{-tG}}^{A_l})(x)\leq 1.$$
Clearly,
$$({\rm id}_{\mathcal{N}}\otimes\cI_{e^{-tG}}^{A_l})(x)\to ({\rm id}_{\mathcal{N}}\otimes\cI_{e^{-tG}}^A)(x)$$
in $L_2(\mathcal{N} \otimes \mathcal{M})$ and  therefore  in measure. Hence,
$$0\leq ({\rm id}_{\mathcal{N}}\otimes\cI_{e^{-tG}}^A)(x)\leq 1.$$
Thus, $(\cI_{e^{-tG}}^A)_{t\geq0}$ is completely positive. Since $( \cI_{e^{-tG}}^A )_{t\geq0}$ is obviously unital, the condition \eqref{markovi} follows.

By assumption, $\cI_{e^{-tG}}^{A_l}$ is self-adjoint on $L_2(\mathcal{M}).$ Clearly, $\cI_{e^{-tG}}^{A_l}\to \cI_{e^{-tG}}^A$ strongly. Therefore, $\cI_{e^{-tG}}^A$ is self-adjoint on $L_2(\mathcal{M}).$ This yields the condition \eqref{markovii}. The condition \eqref{markoviii} is obvious.
\end{proof}

\begin{proof}[Proof of Theorem \ref{bmo boundedness} \eqref{bmoi}] If $A$ has finite spectrum, then the assertion follows by Lemma \ref{transference lemma} and the fact that the Heat semi-group is Markov. For generic $A,$ the assertion follows by Lemma \ref{approximate markov}.
\end{proof}

\subsection*{Proof of Theorem \ref{bmo boundedness} \eqref{bmoii}} $\quad$ \\

\noindent For $A=A^*\in\mathcal{M},$ let $A_l=l^{-1}\lfloor lA\rfloor$ for $l\geq 1.$

\begin{lem}\label{5 convergence lemma} Let $(\cM, \tau)$ be a finite von Neumann algebra. Let $G:\mathbb{R}^2\to\mathbb{R}$ be a continuous function such that $(\cI_{e^{-tG}}^A)_{t\geq0}$ is Markov for every $A=A^*\in\cM.$ We have as $l \rightarrow \infty$ that,
$$\cI^{A_l}_{e^{-tG}}\Big(\cI_{f^{[1]}}^{A_l}(x)^{\ast}\cI_{f^{[1]}}^{A_l}(x)\Big)\to \cI^A_{e^{-tG}}\Big(\cI_{f^{[1]}}^A(x)^{\ast}\cI_{f^{[1]}}^A(x)\Big), $$
$$\cI^{A_l}_{e^{-tG}}\Big(\cI_{f^{[1]}}^{A_l}(x)\Big)\to \cI^A_{e^{-tG}}\Big(\cI_{f^{[1]}}^A(x)\Big), $$
in measure for every $x\in L_2(\cM).$
\end{lem}
\begin{proof}
Denote, for brevity,
$$y_l=\cI_{f^{[1]}}^{A_l}(x),\quad y=\cI_{f^{[1]}}^A(x).$$
We have that $y_l\to y$ in $L_2$-norm and, therefore, $y_l^{\ast}y_l\to y^{\ast}y$ in $L_1$-norm as $l\to\infty.$ We have,
\begin{equation}\label{5conv1}
\cI^{A_l}_{e^{-tG}}(y_l^*y_l)=\cI^{A_l}_{e^{-tG}}(y_l^{\ast}y_l-y^{\ast}y)+\cI^{A_l}_{e^{-tG}}(y^*y).
\end{equation}
Since every Markov semi-group is an  $L_1$-contraction, it follows that
\begin{equation}\label{5conv2}
\cI^{A_l}_{e^{-tG}}(y_l^{\ast}y_l-y^{\ast}y)\to0,\quad l\to\infty,
\end{equation}
in the $L_1$-norm and, therefore in measure. 	

Let $z \in L_1(\cM)$ be arbitrary and fix $\epsilon>0$. Recall that  $L_2(\cM)$ is dense in $L_1(\cM)$, \cite[Theorem IX.2.13]{TakII}. Therefore  choose a decomposition $z = z_1+z_2$ such that $\|z_1\|_1<\epsilon$ and such that $z_2\in L_2(\cM)$.   We have,
\[\cI^{A_l}_{e^{-tG}}(z)-\cI^A_{e^{-tG}}(z)=\Big(\cI^{A_l}_{e^{-tG}}(z_2)-\cI^A_{e^{-tG}}(z_2)\Big)+\cI^{A_l}_{e^{-tG}}(z_1)-\cI^A_{e^{-tG}}(z_1).\]
Clearly,
\[\cI^{A_l}_{e^{-tG}}(z_2)-\cI^A_{e^{-tG}}(z_2)\to0,\quad l\to\infty,\]
in $L_2$-norm. Hence, there exists $l(\epsilon)$ such that, for $l>l(\epsilon),$
$$\|\cI^{A_l}_{e^{-tG}}(z_2)-\cI^A_{e^{-tG}}(z_2)\|_1\leq\|\cI^{A_l}_{e^{-tG}}(z_2)-\cI^A_{e^{-tG}}(z_2)\|_2<\epsilon.$$
Since,
\[\|\cI^{A_l}_{e^{-tG}}(z_1)\|_1<\|z_1\|_1\leq\epsilon,\quad \|\cI^A_{e^{-tG}}(z_1)\|_1<\|z_1\|_1\leq\epsilon,	\]
it follows that
\[\|\cI^{A_l}_{e^{-tG}}(z )-\cI^A_{e^{-tG}}(z )\|_1<3\epsilon,\quad l\geq l(\epsilon).\]
So we conclude that for $z \in L_1(\cM)$ we have
\begin{equation}\label{5conv3}
\cI^{A_l}_{e^{-tG}}(z)\to \cI^A_{e^{-tG}}(z),\quad l\to\infty,
\end{equation}
in $L_1$-norm. Applying this to $z = y^\ast y$ and combining this with \eqref{5conv1} and \eqref{5conv2}, we infer the first assertion. The second (easier) assertion follows as the convergence actually holds in $L_2$-norm.
\end{proof}

\begin{lem}\label{l2circ lemma} If $x\in L_2(\cM),$ then $\cI_{f^{[1]}}^A(x)\in L_2^{\circ}(\cM).$
\end{lem}
\begin{proof}
Let $\mathcal{D}_A$ be the von Neumann algebra generated by the spectral projections of $A$. Let $\mathcal{D}'  = \mathcal{D}_A' \cap \cM$ be its relative commutant with trace preserving conditional expectation $\mathbb{E}_{\mathcal{D}'}: \cM \rightarrow \mathcal{D}$.
If $z\in L_2(\cM),$ then
$$\cI^A_{e^{-tF}}(z)\to \mathbb{E}_{\mathcal{D}'}(z),\quad t\to\infty,$$
in measure. Therefore, $z\in L_2^{\circ}(\cM)$ if and only if $\mathbb{E}_{\mathcal{D}'}(z)=0.$ 		
We claim that $\mathbb{E}_{\mathcal{D}'}(\cI_{f^{[1]}}^A(x))=0.$ Set
$$p_{m,k}=\chi_{[\frac{k}{m},\frac{k+1}{m})}(A),\quad x_m=\sum_{k\in\mathbb{Z}}p_{m,k}xp_{m,k}.$$
We have
$$\cI_{f^{[1]}}^A\Big(p_{m,k}xp_{m,l}\Big)=p_{m,k}\cdot \cI_{f^{[1]}}^A(x)\cdot p_{m,l}.$$
Therefore,
$$\mathbb{E}_{\mathcal{D}'}\Big(\cI_{f^{[1]}}^A\Big(p_{m,k}xp_{m,l}\Big)\Big)=p_{m,k}\cdot \mathbb{E}_{\mathcal{D}'}(\cI_{f^{[1]}}^A(x))\cdot p_{m,l}.$$
If $k\neq l,$ then
$$\mathbb{E}_{\mathcal{D}'}\Big(\cI_{f^{[1]}}^A\Big(p_{m,k}xp_{m,l}\Big)\Big)=p_{m,k}\cdot p_{m,l}\cdot \mathbb{E}_{\mathcal{D}'}(\cI_{f^{[1]}}^A(x))=0.$$
Therefore,
\[
\begin{split}
& \mathbb{E}_{\mathcal{D}'}(\cI_{f^{[1]}}^A(x))=\sum_{k,l\in\mathbb{Z}}\mathbb{E}_{\mathcal{D}'}\Big(\cI_{f^{[1]}}^A\Big(p_{m,k}xp_{m,l}\Big)\Big)  \\
= & \sum_{k\in\mathbb{Z}}\mathbb{E}_{\mathcal{D}'}\Big(\cI_{f^{[1]}}^A\Big(p_{m,k}xp_{m,k}\Big)\Big)= \mathbb{E}_{\mathcal{D}'}(\cI_{f^{[1]}}^A(x_m)).
\end{split}
\]
As $m\to\infty,$ we have convergence in measure
$$x_m\to \mathbb{E}_{\mathcal{D}'} (x),\quad \cI_{f^{[1]}}^A(x_m)\to \cI_{f^{[1]}}^A( \mathbb{E}_{\mathcal{D}'}(x))=0.$$
This concludes the proof.
\end{proof}

\begin{proof}[Proof of Theorem \ref{bmo boundedness} \eqref{bmoii}] By the first equality of Lemma \ref{transference lemma} we see that $\varphi_{A_l}$ maps $\BMO(\cM,\mathcal{I}^{A_l})$ to $\BMO(L_{\infty}(\mathbb{R}^2) \otimes\cM,\mathcal{S}\otimes{\rm id}_{\cM})$ isometrically with $\cS$ the Heat semi-group. By the second equality of Lemma \ref{transference lemma} we then further have,
\[\begin{split}
\|\cI_{f^{[1]}}^{A_l}(x)\|_{\BMO(\cM,\mathcal{I}^{A_l})}= & \|\varphi_{A_l} \circ  \cI_{f^{[1]}}^{A_l}(x)    \|_{\BMO(\cM,\mathcal{I}^{A_l})}\\
= & \|(m_0(\nabla)\otimes{\rm id}_{\mathcal{M}})(\varphi_{A_l}(x))\|_{\BMO(L_{\infty}(\mathbb{R}^2) \otimes\cM,\mathcal{S}\otimes{\rm id}_{\cM})}.
\end{split}
\]
As $\varphi_{A_l}(x)$ is trigonometric, by Proposition \ref{Prop=CBMulti},
\[
\|(m_0(\nabla)\otimes{\rm id}_{\mathcal{M}})(\varphi_{A_l}(x))\|_{\BMO(L_{\infty}(\mathbb{R}^2) \otimes\cM,\mathcal{S}\otimes{\rm id}_{\cM})} \leq c_{abs} \Vert \varphi_{A_l}(x) \Vert_\infty = c_{abs} \Vert x \Vert_\infty.
\]
 Therefore, we have
$$\|\cI_{f^{[1]}}^{A_l}(x)\|_{\BMO(\cM,\mathcal{I}^{A_l})}\leq c_{{\rm abs}}\|x\|_{\infty}.$$
Thus, for every $t\geq0,$ we have
$$-c_{{\rm abs}}^2\|x\|_{\infty}^2\leq B_l(t)\leq c_{{\rm abs}}^2\|x\|_{\infty}^2,$$
where
$$B_l(t)=\cI^{A_l}_{e^{-tF}}\Big(\cI_{f^{[1]}}^{A_l}(x)^{\ast}\cI_{f^{[1]}}^{A_l}(x)\Big)-\cI^{A_l}_{e^{-tF}}\Big(\cI_{f^{[1]}}^{A_l}(x)\Big)^{\ast}\cI^{A_l}_{e^{-tF}}\Big(\cI_{f^{[1]}}^{A_l}(x)\Big).$$
By Lemma \ref{5 convergence lemma}, we have $B_l(t)\to B(t)$ in measure as $l\to\infty.$ Here,
$$B(t)=\cI^A_{e^{-tF}}\Big(\cI_{f^{[1]}}^A(x)^{\ast}\cI_{f^{[1]}}^A(x)\Big)-\cI^A_{e^{-tF}}\Big(\cI_{f^{[1]}}^A(x)\Big)^{\ast}\cI^A_{e^{-tF}}\Big(\cI_{f^{[1]}}^A(x)\Big).$$
Therefore,
$$-c_{{\rm abs}}^2\|x\|_{\infty}^2\leq B(t)\leq c_{{\rm abs}}^2\|x\|_{\infty}^2$$
for every $t\geq0.$ In other words,
$$\|\cI_{f^{[1]}}^A(x)\|_{\BMO(\cM,\cI^A)}\leq c_{{\rm abs}}\|x\|_{\infty}.$$
 By Lemma \ref{l2circ lemma}, we also have $\cI_{f^{[1]}}^A(x)\in L_2^{\circ}(\cM).$ A combination of this fact and the norm estimate complete the proof.
\end{proof}

We shall need the following auxiliary lemma in the next section.
\begin{lem}\label{Lem=Fix}
Suppose that $A$ has finite spectrum. We have that $\BMO(\mathcal{M}, \cI^A) = \BigBMO(\mathcal{M}, \cI^A)$ as vector spaces with equality of norms.
\end{lem}
\begin{proof}
We have an equality \cite[Proof of Lemma 1.3]{JMP} for $f \in L_\infty(\mathbb{R}^n) \otimes \cM$,
\[
\begin{split}
& (e^{t \Delta}\otimes\id_{\cM})  (  f^\ast f)-(e^{t \Delta}\otimes\id_{\cM})  (f)^\ast  (e^{t \Delta}\otimes\id_{\cM})  (f) \\
= &(e^{t \Delta}\otimes\id_{\cM}) \left( \vert f -  (e^{t \Delta}\otimes\id_{\cM})  (f) \vert^2 \right).
\end{split}
\]

For $x\in\cM$ set $f=\varphi_A(x).$ We get by Lemma \ref{transference lemma},
\[
\begin{split}
\Vert x \Vert_{\BigBMO(\cM, \mathcal{I}^A)} =&
\sup_{t>0}\Big\|(e^{t \Delta}\otimes\id_{\cM}) \left( \vert f -  (e^{t \Delta}\otimes\id_{\cM})  (f) \vert^2 \right)\Big\|_{\infty}^{\frac12}\\
=&\sup_{t>0}\Big\|(e^{t \Delta}\otimes\id_{\cM})  (  f^\ast f)-(e^{t \Delta}\otimes\id_{\cM})  (f)^\ast  (e^{t \Delta}\otimes\id_{\cM})  (f)\Big\|_{\infty}^{\frac12}\\
= &\Vert x \Vert_{\BMO(\cM, \mathcal{I}^A)}.
\end{split}
\]
\end{proof}

\section{Conclusions for BMO-estimates for commutators} \label{Sect=Conclusions}

We now collect several results in perturbation theory of commutators as a consequence of  Theorem  \ref{bmo boundedness}. In particular we recover the main results from \cite{PotapovSukochevActa} and \cite{CMPS}. We in fact improve of them in terms of BMO-estimates.

As before  we fix a Lipschitz function $f: \mathbb{R} \rightarrow \mathbb{R}$ and we assume that $\Vert f' \Vert_\infty \leq 1.$
 We set,
\[
\psi(\lambda, \mu) = \lambda - \mu, \qquad \psi_f(\lambda, \mu) = f(\lambda) - f(\mu).
\]
Note that $\cI_\psi^A(x) = Ax - xA = [A,x]$ and $\cI_{\psi_f}^A(x) = [f(A), x]$ for $A \in \cM$ self-adjoint. We start with the following corollary.

\begin{cor} \label{Cor=CommutatorBMO}
In the setting of Theorem \ref{bmo boundedness}, there exists a constant $c_{abs}$ such that for every $A \in \cM$  self-adjoint and every $x \in \cM$ we have,
\[\Vert [f(A), x] \Vert_{\BMO_{\cI^A}} \leq  c_{abs} \Vert [A, x] \Vert_\infty.\]
\end{cor}

\begin{proof}[Proof of Corollary \ref{Cor=CommutatorBMO}]
We have,
\[[f(A), x] = \cI_{\psi_f}^{A}(x) = (\cI_{f^{[1]}}^{A} \circ \cI_{\psi}^A)(x) = \cI_{f^{[1]}}^{A}([A, x]).\]
Hence,   by Theorem \ref{bmo boundedness},
\[\Vert [f(A), x] \Vert_{\BMO_{\cI^A}} =\Vert  \cI_{f^{[1]}}^{A}([A, x]) \Vert_{\BMO_{\cI^A}}\leq \Vert \cI_{f^{[1]}}^{A}: \cM \rightarrow\BMO_{\cI^A} \Vert \Vert [A, x] \Vert_\infty.	\]
\end{proof}

\begin{rmk}
For a general von Neumann algebra $\cM$ one cannot define a canonical Markov semi-group without further structure. This is why in Corollary \ref{Cor=CommutatorBMO} the semi-group depends on the  self-adjoint operator $A \in \cM$ and the Lipschitz function $f$ and we believe this is the suitable  end-point estimate. After interpolation the dependence of $A$ and $f$ vanishes in Theorem \ref{Thm=Commutator} and we obtain best constant estimates.
\end{rmk}

Next, through our BMO approach we collect many optimal results in perturbation theory.
 Firstly we retrieve the main result of \cite{CMPS}.

\begin{thm} \label{Thm=DOIbound}
In the setting of Theorem \ref{bmo boundedness}, let $A \in \cM$ be self-adjoint. 	There exists a constant $c_{abs}$ such that for every $1 < p < \infty$,
\[\Vert \cI_{f^{[1]}}^{A}: L_p(\cM) \rightarrow L_p(\cM) \Vert \leq c_{abs} \frac{p^2}{p-1}.\]
\end{thm}
\begin{proof} Setting as before $A_l=\frac1l\lfloor l A\rfloor,$ we infer from Lemma \ref{Lem=Fix} that $\BMO(\mathcal{M}, \mathcal{I}^{A_l}) = \BigBMO(\mathcal{M}, \mathcal{I}^{A_l}).$

By Theorem \ref{bmo boundedness}  and its proof we have
\[\Vert \cI_{f^{[1]}}^{A_l}: \cM \rightarrow \BMO(\cM, \cI^{A_l}) \Vert \leq \Vert m_0(\nabla): L_\infty \rightarrow \BMO_{\cS} \Vert_{cb}.\]
Also,
$$\Vert \cI_{f^{[1]}}^{A_l}: L_2(\cM) \rightarrow L_2(\cM)   \Vert  \leq \Vert f' \Vert_\infty \leq 1.$$
By Theorem \ref{Thm=DOIStandardMarkovDilationFix} we see that $\cI^{A_l}$ has a standard Markov dilation. Therefore, by Theorem \ref{Thm=JungeMeiInterpolation} for  $2 \leq p < \infty$ we have
\[\Vert \cI_{f^{[1]}}^{A_l}: L_p(\cM) \rightarrow L_p(\cM) \Vert \leq  c_{abs} p.\]
Further, for $x \in \mathcal{M}$ we have $\cI_{f^{[1]}}^{A_l}(x) \rightarrow \cI_{f^{[1]}}^{A}(x)$ in measure as $l \rightarrow \infty$. Hence it follows that also
\[\Vert \cI_{f^{[1]}}^A: L_p(\cM) \rightarrow L_p(\cM) \Vert \leq  c_{abs} p.\]

Next let $1 < p \leq 2$ and let $q$ be conjugate, i.e. $p^{-1} + q^{-1} = 1$. By duality we find that
\[(\cI^{A}_{f^{[1]}})^\ast: L_p(\cM) \rightarrow L_p(\cM),\]
is the extension of the double operator integral   $\cI^{A}_{\bar{f}^{[1]}}.$ So that,
\[\cI^{A}_{f^{[1]}}: L_p(\cM) \rightarrow L_p(\cM)  = (\cI^{A}_{\overline{f}^{[1]}}: L_q(\cM) \rightarrow L_q(\cM) )^\ast	\]
is bounded  with $\Vert \cI^{A}_{f^{[1]}}: L_p(\cM) \rightarrow L_p(\cM)  \Vert \leq c_{abs} (p-1)^{-1}$. 	
\end{proof}

\begin{thm}\label{Thm=Commutator}	
In the setting of Theorem \ref{bmo boundedness}, there exists an absolute constant $c_{abs}$ such that for any operators $A \in \cM$ self-adjoint and $x \in \cM$, and any $1 < p < \infty$, we have
\[\Vert [f(A), x] \Vert_{p} \leq c_{abs} \frac{ p^2}{p -1} \Vert [A, x] \Vert_\infty.\]
\end{thm}
\begin{proof}
We derive the proof from Theorem \ref{Thm=DOIbound} as in Corollary \ref{Cor=CommutatorBMO}. We find that for $x \in \cM$,
\begin{equation}\label{Eqn=ComEst}
\begin{split}
& \Vert [f(A), x] \Vert_{p} =\Vert  \cI^{A}_{f^{[1]}}([A, x]) \Vert_{p} \\
\leq & \Vert \cI^{A}_{f^{[1]}}: L_p \rightarrow L_p \Vert \Vert [A, x] \Vert_p \leq c_{abs} \frac{p^2}{p-1}  \Vert [A, x] \Vert_p.
\end{split}
\end{equation}
\end{proof}

\begin{cor}\label{Cor=LipEst}
There exists a constant $c_{abs}$ such that for any  self-adjoint operators $B,C \in \cM$ we have
\[\Vert f(B) - f(C) \Vert_{p} \leq   c_{abs} \frac{p^2}{p-1} \Vert B - C \Vert_p.\]
\end{cor}
\begin{proof}
Apply Theorem \ref{Thm=Commutator} to
\[x=
\left( \begin{array}{cc}
0 & 1 \\
1 & 0
\end{array}\right), \qquad A =
\left( \begin{array}{cc}
B & 0 \\
0 & C
\end{array} \right).
\]
As
\[[A, x] = \left( \begin{array}{cc}
0 & B-C \\
C-B & 0
\end{array} \right),
\]
we find  $\Vert [A, x] \Vert_p =   2^{\frac{1}{p}} \Vert B - C \Vert_p$ and similarly  $\Vert [f(A), x] \Vert_p = 2^{\frac{1}{p}} \Vert f(B) - f(C) \Vert_p$.

Theorem \ref{Thm=Commutator} gives
\[2^{\frac{1}{p}}\Vert f(B) - f(C) \Vert_p = \Vert [f(A), x] \Vert_p \leq c_{abs} \frac{p^2}{p-1}  \Vert [A, x] \Vert_p = 2^{\frac{1}{p}} c_{abs} \frac{p^2}{p-1} \Vert B - C \Vert_p.\]
\end{proof}

As another  corollary we get a proof of the Aleksandrov-Peller results in \cite[Theorem 11.4]{AleksandrovPeller}.

\begin{thm}\label{Thm=LogEstimate}
There exists a constant $c_{abs}$ such that for any two self-adjoint operators $A, B \in \cB(\mathbb{C}^n)$ and any Lipschitz function $f: \mathbb{R} \rightarrow \mathbb{C}$ we have
\[\Vert f(B) - f(C) \Vert_\infty \leq C_{abs} (1 + \log(n)) \Vert B - C \Vert_\infty.\]
\end{thm}
\begin{proof}
Let $\cS_p^n$ be the Schatten class associated with $\cB(\mathbb{C}^n)$. We have that $\cB(\mathbb{C}^n) \subseteq \cS_p^n$ contractively. The converse inclusion $\cS_p^n \subseteq \cB(\mathbb{C}^n)$  has norm at most $n^{\frac{2}{p}}$ by complex interpolation between $p = 1$ and $p = \infty$.

We find that for $\log(2) \leq p < \infty$,
\[\begin{split}
\Vert f(B) - f(C) \Vert_\infty \leq      \Vert f(B) - f(C) \Vert_p\leq c_{abs} \: p  \Vert B - C \Vert_p \leq c_{abs} p n^{\frac{2}{p}} \Vert B - C \Vert_\infty.
\end{split}\]
Now for $n \geq 2$ take $p = \log(n)$ so that we get,
\[\Vert f(B) - f(C) \Vert_\infty \leq c_{abs} \log(n)  e^{\frac{2}{p} \log(n)}  \Vert B - C \Vert_\infty= c_{abs} \log(n)  e^{2}  \Vert B - C \Vert_\infty\]
This yields the theorem as in case $n = 1$ it is obvious.
\end{proof}

\section{Estimates for vector-valued double operator integrals}\label{Sect=VectTrans}

 \subsection{Assumptions and statements}

 As before we let $(\cM, \tau)$ be a finite von Neumann algebra. Throughout the entire section we fix a finite von Neumann algebra $\cN$ whose trace shall not be used explicitly. We write $\cS^n$ for the Heat semi-group on $\mathbb{R}^n$ to stipulate the dimension. We let $C(\mathbb{R}^n, \cN)$ be the space of norm continuous functions $\mathbb{R}^n \rightarrow \cN$.

 Let ${\bf A}=(A_1,\ldots,A_n)$ be an $n$-tuple of commuting self-adjoint elements in $\cM.$ Let $E^{{\bf A}}:\mathbb{R}^n \rightarrow \cM$ be their joint spectral measure.

 \begin{dfn}\label{ja semigroup}
 	The semi-group $\mathcal{J}^{{\bf A}}:\cM\otimes M_2(\mathbb{C})\to \cM\otimes M_2(\mathbb{C})$ is defined by the formula
 	$$\mathcal{J}^{{\bf A}}_t(x\otimes e_{ij})=\cI_{e^{-tF_{ij}}}^{{\bf A}}(x)\otimes e_{ij},\quad x\in\cM,$$
 	where
 	$$F_{ij}(\lambda,\mu)=
 	\begin{cases}
 	|\lambda-\mu|^2,& i=j,\\
 	|\lambda|^2+|\mu|^2,& i\neq j.
 	\end{cases}
 	$$
 \end{dfn}

 Let $L:L_{\infty}(\mathbb{R}^n,\cN)\to L_{\infty}(\mathbb{R}^{2n},\cN)$ be defined by the formula
 \begin{equation}\label{Eqn=Lmap}
 (Lh)(t,s)=\int_0^1h(\theta s+(1-\theta)t)d\theta.
 \end{equation}
  We need the following H\"ormander-Mikhlin-type condition.

 \begin{cond}\label{hormikh} The function $h\in C(\mathbb{R}^n,\cN)$ is a compactly supported   $C^{n+2}$-function  such that
 	\begin{equation}\label{Eqn=SmoothEnough}
 	\|h\|_{HM_n}\stackrel{{\rm def}}{=}\max_{0\leq|\alpha|\leq n+2}\sup_{t\neq 0} \Vert t\Vert_2^{|\alpha|}\Vert (\partial_\alpha h)(t)\Vert_{\cN} \leq 1.
 	\end{equation}
 \end{cond}

The following Theorem \ref{vector doi thm} is the main result we prove in this sections from which we derive vector valued commutator estimates.

 \begin{thm}\label{vector doi thm} Let ${\bf A}=(A_1,\ldots,A_n)$ be an $n$-tuple of commuting self-adjoint elements in $\cM.$ The following statements hold.
 	\begin{enumerate}[{\rm (i)}]
 		\item\label{vectordoii} The semi-group $\mathcal{J}^{{\bf A}}$ in Definition \ref{ja semigroup} is Markov.
 		\item\label{vectordoiii} If $h\in C(\mathbb{R}^n,\cN)$ satisfies the Condition \ref{hormikh}   and $h(0) = 0$, then
 		$$\Big\|\cI_{Lh}^{{\bf A}}(x)\otimes e_{12}\Big\|_{\BMO(\cN \otimes \cM\otimes M_2(\mathbb{C}),{\rm id}_{\cN}\otimes\mathcal{J}^{{\bf A}})}\leq c_{n}   \|x\|_{\cM},\quad x\in\cM.$$
 	\end{enumerate}
 \end{thm}

 \subsection{Transference of multipliers}

 Let again $e_s \in L_\infty(\mathbb{R}^n)$ be given by $e_s(t) = e^{i \langle s, t \rangle}.$
When ${\bf A}$ has finite spectrum, define unitary operators $U,V\in L_{\infty}(\mathbb{R}^n)\otimes L_{\infty}(\mathbb{R}^n)\otimes\cM$ by the formula
 $$U=\int_{\mathbb{R}^n}e_s\otimes 1_{L_{\infty}(\mathbb{R}^n)}\otimes dE^{{\bf A}}(s),\quad V=\int_{\mathbb{R}^n}1_{L_{\infty}(\mathbb{R}^n)}\otimes e_{-s}\otimes dE^{{\bf A}}(s).$$
 Set
 $$W=U\otimes e_{11}+V\otimes e_{22},$$
 and further
 $$\pi_{{\bf A}}(x)=W (1_{L_{\infty}(\mathbb{R}^n)}\otimes 1_{L_{\infty}(\mathbb{R}^n)}\otimes x)  W^{\ast},\quad x\in\cM\otimes M_2(\mathbb{C}).$$
The map
\[
{\rm corner}:\cM\to\cM\otimes M_2(\mathbb{C})
\]
is defined by the  formula $x\to x\otimes e_{12}.$

\begin{prop}\label{transference vector} Let ${\bf A}=(A_1,\ldots,A_n)$ be an $n$-tuple of commuting self-adjoint elements in $\cM$ with finite spectrum.
\begin{enumerate}[{\rm (i)}]
\item for every $t\geq0,$ we have
$$(S_t^{2n}\otimes {\rm id}_{\cM\otimes M_2(\mathbb{C})})\circ \pi_{{\bf A}}=\pi_{{\bf A}}\circ\mathcal{J}_t^{{\bf A}}.$$
\item for every $k\in C(\mathbb{R}^{2n},\cN),$ we have
$$(k(\nabla_{\mathbb{R}^{2n}})\otimes {\rm id}_{\cM\otimes M_2(\mathbb{C})})\circ \pi_{{\bf A}}\circ{\rm corner}=({\rm id}_{\cN}\otimes\pi_{{\bf A}})\circ({\rm id}_{\cN}\otimes{\rm corner})\circ\mathcal{I}_k^{{\bf A}}.$$
\end{enumerate}
\end{prop}
\begin{proof} If $x\in\cM,$ then
\[\begin{split}
\pi_{{\bf A}}(x\otimes e_{11})= & U x U^\ast \otimes e_{11}= \int_{\mathbb{R}^n}\int_{\mathbb{R}^n}e_{s-u}\otimes 1_{L_{\infty}(\mathbb{R}^n)}\otimes dE^{{\bf A}}(s)xdE^{{\bf A}}(u)\otimes e_{11}.
\end{split}
\]
Therefore,
\[\begin{split}
& \Big(S_t^{2n}\otimes {\rm id}_{\cM\otimes M_2(\mathbb{C})}\Big)(\pi_{{\bf A}}(x\otimes e_{11}))\\
= &\int_{\mathbb{R}^n}\int_{\mathbb{R}^{n}}e^{-t|s-u|^2}e_{s-u}\otimes 1_{L_{\infty}(\mathbb{R}^n)}\otimes dE^{{\bf A}}(s)xdE^{{\bf A}}(u)\otimes e_{11} \\
= & \pi_{{\bf A}}\Big(\int_{\mathbb{R}^n}\int_{\mathbb{R}^n}e^{-t|s-u|^2}dE^{{\bf A}}(s)xdE^{{\bf A}}(t)\otimes e_{11}\Big)\\
=&(\pi_{{\bf A}}\circ\mathcal{J}_t^{{\bf A}})(x\otimes e_{11}).
\end{split}\]
Also,
$$\pi_{{\bf A}}(x\otimes e_{12})= U x V^\ast \otimes e_{12} = \int_{\mathbb{R}^n}\int_{\mathbb{R}^n}e_s\otimes e_t\otimes dE^{{\bf A}}(s)xdE^{{\bf A}}(t)\otimes e_{12}.$$
Therefore,
\[\begin{split}
& \Big(S_t^{2n}\otimes {\rm id}_{\cM\otimes M_2(\mathbb{C})}\Big)(\pi_{{\bf A}}(x\otimes e_{12})) \\
= &\int_{\mathbb{R}^n}\int_{\mathbb{R}^n}e^{-t|s|^2}e_s\otimes e^{-t|u|^2}e_u\otimes dE^{{\bf A}}(s)xdE^{{\bf A}}(u)\otimes e_{12} \\
= & \pi_{{\bf A}}\Big(\int_{\mathbb{R}^n}\int_{\mathbb{R}^n}e^{-t|s|^2-t|u|^2}dE^{{\bf A}}(s)xdE^{{\bf A}}(u)\otimes e_{12}\Big)\\
= & (\pi_{{\bf A}}\circ\mathcal{J}_t^{{\bf A}})(x\otimes e_{12}).
\end{split}\]
The argument for $x\otimes e_{21}$ and for $x\otimes e_{22}$ goes {\it mutatis mutandi}. This proves the first assertion.
 	
We have,
$$\pi_{{\bf A}}(x\otimes e_{12})=\int_{\mathbb{R}^n}\int_{\mathbb{R}^n}e_s\otimes e_t\otimes dE^{{\bf A}}(s)xdE^{{\bf A}}(t)\otimes e_{12}.$$
Therefore,
\[\begin{split}
& \Big(k(\nabla_{\mathbb{R}^{2n}})\otimes {\rm id}_{\cM\otimes M_2(\mathbb{C})}\Big)(\pi_{{\bf A}}(x\otimes e_{12})) \\
= & \int_{\mathbb{R}^n}\int_{\mathbb{R}^n}k(s,u)\otimes e_s\otimes e_u\otimes dE^{{\bf A}}(s)xdE^{{\bf A}}(u)\otimes e_{12} \\
= & ({\rm id}_{\cN}\otimes\pi_{{\bf A}})\Big(\int_{\mathbb{R}^n}\int_{\mathbb{R}^n}k(s,u)\otimes dE^{{\bf A}}(s)xdE^{{\bf A}}(u)\otimes e_{12}\Big) \\
= & ({\rm id}_{\cN}\otimes\pi_{{\bf A}})\Big(\cI_k^{{\bf A}}(x)\otimes e_{12}\Big).
\end{split}
\]
This proves the second assertion.
\end{proof}

 \subsection{Smooth vector valued multipliers}

 Fix a convolution kernel  $K: \mathbb{R}^n \backslash \{ 0 \} \rightarrow \cN$.
 Assume $K$ determines a convolution operator by the principal value integral,
 \[
 (K \ast g)(x) = \int_{\mathbb{R}^{n}} K(y) g(x-y) dy, \qquad x \in \mathbb{R}^{n},
 \]
 and $g: \mathbb{R}^n \rightarrow \mathbb{C}$ smooth and compactly supported; the domain of $K \ast$ will be extended shortly. We shall also write $K$ for $K\ast$, the convolution operator (Calder\'on-Zygmund operator).
 We say that a function $K: \mathbb{R}^n \backslash \{ 0 \} \rightarrow \cN$ satisfies the H\"ormander-Mikhlin condition if,
 \begin{equation}\label{Eqn=SmoothnessCondition}
 \sup_{x \in \mathbb{R}^{2n}}  \int_{y \in \mathbb{R}^{2n}, \Vert y \Vert_2 > 2 \Vert x \Vert_2} \Vert  K(x-y) - K(y)   \Vert_{\cN}   dy < \infty.
 \end{equation}
 Note that if $h\in C(\mathbb{R}^n,\cN)$ is integrable  it has a Fourier transform $\widehat{h}: \mathbb{R}^n \rightarrow \cN$ that is uniquely determined by  $\omega \circ \widehat{h} = \widehat{\omega \circ h}$ for all $\omega \in \cN_\ast$.
 The following is essentially proved in \cite[Lemma 2.3 and Lemma 3.3]{JMP}; we explain how it can be derived.

 \begin{prop}\label{Prop=HormanderMikhlin1}
 	If $h\in C(\mathbb{R}^n,\cN)$ satisfies Condition \ref{hormikh}   and $h(0) = 0$, then
 	$$h(\nabla_{\mathbb{R}^n})\otimes {\rm id}_{\cM}: L_{\infty}(\mathbb{R}^n)\otimes\cM\to\BMO(\cN\otimes L_{\infty}(\mathbb{R}^n)\otimes\cM,{\rm id}_{\cN}\otimes\mathcal{S}^n\otimes{\rm id}_{\cM})$$
 	and its norm is bounded by an absolute constant $c^n_{abs}$ only depending on the dimension $n$.
 \end{prop}
 \begin{proof}
 Let $\omega \in \cN_{\ast}$ with $\Vert \omega \Vert = 1$ then $h^\omega  := \omega \circ h$ is a $C^{n+2}$-function whose associated  Fourier transform  is given by $K^\omega = \omega \circ K$.  We still have $\sup_{\vert \alpha \vert \leq n +2} \Vert \xi \Vert_2^{\vert \alpha \vert} \Vert  \partial_\alpha \omega \circ h \Vert_{\cN} \leq 1$.
 	The proof  of \cite[Lemma 3.3]{JMP} (more precisely, the statement in its first line)   shows that we have the gradient estimate,
 	$\vert \Vert s \Vert^{+n+1} (\nabla K^\omega)(s) \vert \leq   c_{abs}^n,  s \in \mathbb{R}^{n} \backslash \{ 0 \}$,
 	for an absolute constant $c_{abs}^n$    indepenedent of $\omega$.  Taking the supremum over all $\omega$ in the unit ball of $\cN_\ast$ concludes that in fact $  \Vert (\nabla K)(s) \Vert_{\cN}  \leq c_{abs}^n  \Vert s \Vert^{-n-1},   s \in \mathbb{R}^{n} \backslash \{ 0 \}$.
 	So that certainly \eqref{Eqn=SmoothnessCondition} holds. So \cite[Lemma 2.3, Condition (ii)]{JMP} 	is fulfilled. Further  \cite[Lemma 2.3, Condition (i)]{JMP} is satisfied by Remark \cite[Remark 2.4]{JMP}. Hence,   \cite[Lemma 2.3]{JMP} gives the result for  the column estimate. The row estimate follows by taking adjoints.   Note that as in Proposition \ref{Prop=CBMulti} the condition $h(0) = 0$ guarantees that $h(\nabla_{\mathbb{R}^n})(1) = 0$ so that $h(\nabla_{\mathbb{R}^n})(f) \in \cN \odot L_\infty^\circ(\mathbb{R}^n)$ with $f \in L_\infty^\circ(\mathbb{R}^n)$ trigonometric.
 \end{proof}

 \begin{prop}\label{Prop=HormanderMikhlin2}
 	If $h\in C(\mathbb{R}^n,\cN)$ satisfies Condition \ref{hormikh}  and $h(0) = 0$, then for $y\in L_{\infty}(\mathbb{R}^{2n}) \otimes \cM$,
 	\[
 	\begin{split}
 	& \Big\|\Big((Lh)(\nabla_{\mathbb{R}^{2n}})\otimes {\rm id}_{\cM_2}\Big)(y)\Big\|_{\BMO(\cN \otimes L_{\infty}(\mathbb{R}^{2n}) \otimes \cM,{\rm id}_{\cN}\otimes\mathcal{S}^{2n}\otimes{\rm id}_{\cM})}\\
 	& \qquad \qquad \qquad \qquad \leq     c_{n}   \|y\|_{L_{\infty}(\mathbb{R}^{2n}) \otimes \cM},
 	\end{split}
 	\]
 \end{prop}
 \begin{proof} Set $h_{\theta}(t,s)=h(\theta s+(1-\theta)t).$ Set $g_{\theta}(t)=h(t\cdot\sqrt{\theta^2+(1-\theta)^2}).$ By definition, we have
 	$$\Big((Lh)(\nabla_{\mathbb{R}^{2n}})\otimes {\rm id}_{\cM}\Big)(y)=\int_0^1\Big(h_{\theta}(\nabla_{\mathbb{R}^{2n}})\otimes {\rm id}_{\cM}\Big)(y)d\theta,$$
 	where the integral is a Bochner integral in $L_2.$ Therefore, we have
 	\[
 	\begin{split}
 	& \Big\|\Big((Lh)(\nabla_{\mathbb{R}^{2n}})\otimes {\rm id}_{\cM}\Big)(y)\Big\|_{\BMO(\cN \otimes L_{\infty}(\mathbb{R}^{2n}) \otimes \cM,{\rm id}_{\cN}\otimes\mathcal{S}^{2n}\otimes{\rm id}_{\cM})} \\
 	\leq &\int_0^1\Big\|\Big(h_{\theta}(\nabla_{\mathbb{R}^{2n}})\otimes {\rm id}_{\cM}\Big)(y)\Big\|_{\BMO(\cN \otimes L_{\infty}(\mathbb{R}^{2n}) \otimes \cM,{\rm id}_{\cN}\otimes\mathcal{S}^{2n}\otimes{\rm id}_{\cM})}d\theta.
 	\end{split}
 	\]
 	By the rotation invariance, we have
 	\[
 	\begin{split}
 	& \Big\|\Big(h_{\theta}(\nabla_{\mathbb{R}^{2n}})\otimes {\rm id}_{\cM}\Big)(y)\Big\|_{\BMO(\cN \otimes L_{\infty}(\mathbb{R}^{2n}) \otimes \cM,{\rm id}_{\cN}\otimes\mathcal{S}^{2n}\otimes{\rm id}_{\cM})} \\
 	= & \Big\|\Big(g_{\theta}(\nabla_{\mathbb{R}^n})\otimes{\rm id}_{L_{\infty}(\mathbb{R}^n)}\otimes {\rm id}_{\cM}\Big)(y)\Big\|_{\BMO(\cN \otimes L_{\infty}(\mathbb{R}^{2n}) \otimes \cM,{\rm id}_{\cN}\otimes\mathcal{S}^{2n}\otimes{\rm id}_{\cM})}.
 	\end{split}
 	\]
 	By Proposition \ref{Prop=HormanderMikhlin1}, we have
 	\[
 	\begin{split}
 	& \Big\|\Big(g_{\theta}(\nabla_{\mathbb{R}^n})\otimes{\rm id}_{L_{\infty}(\mathbb{R}^n)}\otimes {\rm id}_{\cM}\Big)(y)\Big\|_{\BMO(\cN \otimes L_{\infty}(\mathbb{R}^{2n}) \otimes\cM,{\rm id}_{\cN}\otimes\mathcal{S}^{2n}\otimes{\rm id}_{\cM})} \\
 	\leq &  c_{n} \|y\|_{\cM}= c_{n}   \|y\|_{\cM}.
 	\end{split}
 	\]
 	Combining these inequalities, we complete the proof.
 \end{proof}

 \subsection{Proof of Theorem \ref{vector doi thm}}
 In order to explicitly give the proof of Lemma \ref{vector doi finite spectrum} below we single out the following fact.

 \begin{fact}\label{tensor fact} Let $\cM_1$ and $\cM_2$ be von Neumann algebras. Suppose that
 	\begin{enumerate}[{\rm (i)}]
 		\item $\mathcal{T}^1=(T^1_t)_{t\geq0}$ is a semi-group of positive unital operators on $\cM_1;$
 		\item $\mathcal{T}^2=(T^2_t)_{t\geq0}$ is a semi-group of positive unital operators on $\cM_2;$
 		\item $^{\ast}-$monomorphism $\pi:\cM_1\to\cM_2$ is such that
 		$$T_t^1\circ\pi=\pi\circ T_t^2,\quad t\geq0.$$
 	\end{enumerate}
 	We have
 	\begin{enumerate}[{\rm (a)}]
 		\item If $\mathcal{T}^1$ is completely positive, then so is $\mathcal{T}^2$ and
 		$$({\rm id}_{\cN}\otimes\mathcal{T}^1)\circ({\rm id}_{\cN}\otimes\pi)=({\rm id}_{\cN}\otimes\pi)\circ({\rm id}_{\cN}\otimes\mathcal{T}^2).$$
 		\item If $\mathcal{T}^1$ and $\mathcal{T}^2$ are Markov, then
 		$$\|({\rm id}_{\cN}\otimes\pi)(z)\|_{\BMO(\cN \otimes\cM_1,{\rm id}_{\cN}\otimes\mathcal{T}^1)}=\|z\|_{\BMO(\cN  \otimes \cM_2,{\rm id}_{\cN}\otimes\mathcal{T}^2)},\quad z\in\cN \otimes \cM_2.$$
 	\end{enumerate}
 \end{fact}

 \begin{lem}\label{vector doi finite spectrum} The assertion of Theorem \ref{vector doi thm} holds provided that ${\bf A}$ has finite spectrum.
 \end{lem}
 \begin{proof} Let
 	\[
 	\mathcal{M}_1=L_{\infty}(\mathbb{R}^{2n}) \otimes \cM\otimes M_2(\mathbb{C}),\quad \cM_2=\cM\otimes M_2(\mathbb{C}),
 	\]
 	and
 	\[
 	\mathcal{T}^1=\mathcal{S}^{2n}\otimes{\rm id}_{\cM\otimes M_2(\mathbb{C})},\quad \mathcal{T}^2=\mathcal{J}^{{\bf A}},\quad \pi=\pi_{{\bf A}}.
 	\]
 	Denote for brevity
 	$$y=\pi(x\otimes e_{12}),\quad z=\cI_{Lh}^{{\bf A}}(x)\otimes e_{12}.$$
 	By Lemma \ref{transference vector}, we have
 	$$T_t^1\circ\pi=\pi\circ T_t^2,\quad t\geq0.$$
 	Since $\mathcal{T}^1$ is completely positive semi-group, then, by Fact \ref{tensor fact}, so is $\mathcal{T}^2.$ Since ${\bf A}$ has finite spectrum, it follows immediately that $\mathcal{T}^2$ is symmetric and strongly continuous at $0.$ In other words, $\mathcal{T}^2$ is Markov. This proves the first assertion of Theorem \ref{vector doi thm}.
 	
 	By Proposition \ref{Prop=HormanderMikhlin2} (applied to the algebra $\cM_2$), we have
 	\[
 	\begin{split}
 	& \Big\|\Big((Lh)(\nabla_{\mathbb{R}^{2n}})\otimes {\rm id}_{\cM_2}\Big)(y)\Big\|_{\BMO(\cN \otimes\cM_1,{\rm id}_{\cN}\otimes\mathcal{T}_1)} \\
    \leq & c_{n} \|y\|_{\cM_1}= c_{n} \|x\|_{\cM}.
    \end{split}
    \]
 	By Lemma \ref{transference vector}, we have
 	$$\Big((Lh)(\nabla_{\mathbb{R}^{2n}})\otimes {\rm id}_{\cM_2}\Big)(y)=({\rm id}_{\cN}\otimes\pi)(z).$$
 	Therefore, we have
 	$$\Big\|({\rm id}_{\cN}\otimes\pi)(z)\Big\|_{\BMO(\cN \otimes \cM_1,{\rm id}_{\cN}\otimes\mathcal{T}_1)}\leq c_{n} \|x\|_{\cM}.$$
 	By Fact \ref{tensor fact}, we have
 	$$\|z\|_{\BMO(\cN \otimes\cM_2,{\rm id}_{\cN}\otimes\mathcal{T}_2)}\leq c_{n} \|x\|_{\cM}.$$
 	This proves the second assertion of Theorem \ref{vector doi thm}.
 \end{proof}

 \begin{proof}[Proof of Theorem \ref{vector doi thm}]
 	
 	Suppose now ${\bf A}$ is arbitrary. Set ${\bf A}^l=(\frac1l\lfloor l A_1\rfloor,\cdots,\frac1l\lfloor lA_n\rfloor)$.
 By Lemma \ref{vector doi finite spectrum}, $\mathcal{J}^{{\bf A}^l}$ is Markov. Clearly, $\mathcal{J}^{{\bf A}^l}_t(x)\to \mathcal{J}^{{\bf A}}_t(x)$ in $L_2-$norm and  hence in measure  for $x\in\cM$ as $l\to\infty.$ By Lemma \ref{approximate markov fact}  $\mathcal{J}^{{\bf A}}$ is also Markov. This proves the first assertion.
 	
 	We briefly sketch the proof of the second assertion. By Lemma \ref{vector doi finite spectrum}, we have
 	\[
 	\Big\|\cI_{Lh}^{{\bf A}^l}(x)\otimes e_{12}\Big\|_{\BMO(\cN \otimes \cM\otimes M_2(\mathbb{C}),{\rm id}_{\cN}\otimes\mathcal{J}^{{\bf A}^l})}\leq c_{n} \|x\|_{\cM},\quad x\in\cM.
 	\]
 	In other words, we have
 	$$-(c_{n}  \|x\|_{\cM})^2\leq B_l(t)\leq (c_{n} \|x\|_{\cM})^2,$$
 	where
 	$$B_l(t)\stackrel{{\rm def}}{=}\cI^{{\bf A}^l}_{e^{-tF_{22}}}\Big(\cI_{Lh}^{{\bf A}^l}(x)^{\ast}\cI_{Lh}^{{\bf A}^l}(x)\Big)-\Big(\cI^{{\bf A}^l}_{e^{-tF_{12}}}(\cI_{Lh}^{{\bf A}^l}(x))\Big)^{\ast}\Big(\cI^{{\bf A}^l}_{e^{-tF_{12}}}(\cI_{Lh}^{{\bf A}^l}(x))\Big).$$
 	An argument identical to that in Lemma \ref{5 convergence lemma} yields $B_l(t)\to B(t)$ in measure, where
 	$$B(t)\stackrel{{\rm def}}{=}\cI^{{\bf A}}_{e^{-tF_{22}}}\Big(\cI_{Lh}^{{\bf A}}(x)^{\ast}\cI_{Lh}^{{\bf A}}(x)\Big)-\Big(\cI^{{\bf A}}_{e^{-tF_{12}}}(\cI_{Lh}^{{\bf A}}(x))\Big)^{\ast}\Big(\cI^{{\bf A}}_{e^{-tF_{12}}}(\cI_{Lh}^{{\bf A}}(x))\Big).$$
 	Therefore, we have
 	$$- (c_{n}^n  \|x\|_{\cM})^2\leq B(t)\leq (c_{n}  \|x\|_{\cM})^2,$$
 	In other words,
 	$$\Big\|\cI_{Lh}^{{\bf A}}(x)\otimes e_{12}\Big\|_{\BMO(\cN \otimes \cM\otimes M_2(\mathbb{C}),{\rm id}_{\cN}\otimes\mathcal{J}^{{\bf A}})}\leq c_n \|x\|_{\cM},\quad x\in\cM.$$
 \end{proof}

 \section{Vector valued perturbations  and Lipschitz estimates}\label{Sect=Vector}

 In this section we consider   vector valued  commutator estimates.
 Consider a function
 \[
 f: \mathbb{R}^n \rightarrow \cN,
 \]
 which we assume to be  differentiable.
 Shortly, we shall require additional smoothness assumptions on $f$. The function $f$ plays the role of the Lipschitz function in Section \ref{Sect=ComBMO}.
 %
 For a differentiable function $g: \mathbb{R} \rightarrow \mathbb{C}$  and $s \leq t$ we have,
 \[
 g(t) - g(s) =   (t-s)  \int_0^1 g'( (1-\theta)s + \theta t)  d\theta.
 \]
 Therefore taking directional derivatives in the direction of the unit vector $\Vert t -s \Vert_2^{-1}(t-s)$ we find,
 \begin{equation}\label{Eqn=Directional}
 \begin{split}
 f(t) - f(s) = &
 \Vert t - s \Vert_2  \int_0^1 (\nabla \vert_{(1-\theta) s + \theta t } f)  \cdot \frac{t-s}{\Vert t - s \Vert_2} d\theta \\
 = &
 \sum_{k=1}^n \int_0^1 (\partial_k f)((1-\theta) s + \theta t ) (s_k - t_k) d\theta \\
 = & \sum_{k=1}^n (L \partial_k f)(s,t) (s_k - t_k).
 \end{split}
 \end{equation}

 \begin{lem} \label{Lem=Schwartz}
 	Let $c > 0$.
 	\begin{enumerate}[{\rm (i) }]
\item 	There exist Schwartz functions $\varphi_l: \mathbb{R}^n \rightarrow [0,1]$ that are compactly supported with $\varphi_l(\xi) = 1$ for $\Vert \xi \Vert_2  \leq l$ and  $\Vert \xi \Vert_2^{\vert \alpha \vert}  \vert \partial_\alpha \varphi_l  )(\xi) \vert \leq c$ for all  $1 \leq \vert \alpha \vert \leq n+2$.
\item  If $h \in C(\mathbb{R}^n, \cN)$ is a $C^{n+2}$-function that satisfies \eqref{Eqn=SmoothEnough}   then  $(1+c \cdot 2^{n+2} ) ^{-1} \varphi_l h$ satisfies Condition \ref{hormikh}.  	
\end{enumerate}
 \end{lem}
\begin{proof}
	In case $n =1$ and $l=1$ let $\varphi^1_1: \mathbb{R} \rightarrow [0,1]$ be a function satisfying the conditions and then set $\varphi_l^1(\xi) = \varphi_l^1(l^{-1} \xi)$ which proves the lemma for $n=1$. For general $n$ set the rotational invariant function $\varphi_l^n(\xi) = \varphi^1_l(\Vert \xi \Vert_2)$ which are Schwartz and satisfy (i) and (ii).  We have for $\xi_1 \in \mathbb{R}$  that  $\vert (\partial_\alpha \varphi_l  )(\xi_1, 0, \ldots, 0)  \vert 	
	 = \delta_{\vert \alpha \vert = \alpha_1}   \vert  ( \partial_{\alpha_1}  \varphi_l^1)(\xi_1)  \vert \leq  c \Vert \xi \Vert_2^{\alpha_1}$. By rotation of variables this gives $\vert (\partial_\alpha \varphi_l  )(\xi)  \vert \leq c \Vert \xi \Vert_2^{- \vert \alpha \vert}$.	 		 By the Leibniz rule,
	 \[
	 \partial_\alpha ( \varphi_l h) = \sum_{\beta + \gamma = \alpha} c_{\beta, \gamma} (\partial_\beta \varphi_l) \: (\partial_\gamma f),
	 \]
	 for certain combinatorical coefficients $c_{\beta, \gamma} \in \mathbb{N}$ which satisfy  $ \sum_{\beta + \gamma = \alpha} c_{\beta, \gamma} =  2^{\vert \alpha \vert}$. So that,
	 \[
	 \begin{split}
	 & 	  \vert (\partial_\alpha  \varphi_l h)(\xi) \vert \leq
	 \sum_{\beta + \gamma = \alpha} c_{\beta, \gamma}  \vert (\partial_\beta   \varphi_l)(\xi) \vert  \:\vert (\partial_\gamma h)(\xi) \vert \\
	 \leq &
	 \Vert \xi \Vert_2^{-\vert \alpha \vert} + \sum_{\beta + \gamma = \alpha, \beta \not = 0} c \cdot c_{\beta, \gamma}  \Vert \xi \Vert_2^{- \vert \beta \vert} \Vert \xi \Vert_2^{-\vert \gamma\vert}
	  \leq (1+c \cdot 2^{\vert \alpha \vert} ) \Vert \xi \Vert_2^{-\vert \alpha \vert}.
	 \end{split}
	 \]
	 So for all $\vert \alpha \vert \leq n+2$ we obtain that  $\vert (\partial_\alpha  \varphi_l f)(\xi) \vert \leq (1+c \cdot 2^{n+2} )  \Vert \xi \Vert_2^{-\vert \alpha \vert}$, i.e. Condition \ref{hormikh}.	
\end{proof}

 For a function $f \in C(\mathbb{R}^n, \cN)$ and an $n$-tuple $\boldA$ of commuting self-adjoint operators in $\cM$ we define,
 \[
 f(\boldA) = \int_{\mathbb{R}^n} f(\xi) \otimes dE^{\boldA}(\xi) \in \cN \otimes \cM,
 \]
 where $E^{\boldA}$ was the spectral measure of the $n$-tuple $\boldA$. It is the unique element in $\cN \otimes \cM$ such that for every $\omega \in \cN_\ast$ we have
 \[
 (\omega \otimes \id)(f(\boldA) ) = \int_{\mathbb{R}^n} \omega \circ f(\xi)  dE^{\boldA}(\xi) \in \cM.
 \]

 \begin{thm}\label{Thm=VecCom}
 	Let $f: \mathbb{R}^n \rightarrow \cN$ be a $C^{n+3}$-function such that each of the functions $h_k = \partial_k f, k= 1, \ldots, n$ satisfy \eqref{Eqn=SmoothEnough}.
 	There exists a constant $c_{n}$ only depending on the dimension $n$  such that for every $x \in L_2(\cM) \cap L_p(\cM)$ and every $n$-tuple $\boldA = (A_1, \ldots, A_n)$ of commuting self-adjoint operators in $\cM$ we have $[f(\boldA), 1 \otimes x] \in L_p( \cN \otimes \cM  )$. Moreover,
 	\[
 	\Vert  [f(\boldA), 1 \otimes x]  \Vert_{p} \leq
 	c_{n} \frac{p^2}{p-1} \max_k( \Vert h_k \Vert_{HM_n}  ) \sum_{k=1}^n \Vert [A_k, x] \Vert_{p}.
 	\]	
 \end{thm}
 \begin{proof}
  As the theorem is true for the coordinate functions $g_k: \mathbb{R}^n \rightarrow \mathbb{R}: \xi \mapsto \xi_k, 1 \leq k \leq n$, we may replace $f$ by $f - \sum_{k=1}^n (\partial_k f)(0) g_k$  and assume without loss of generality that $(\partial_k f)(0) = 0, 1 \leq k \leq n$.
 	
 	 By Lemma \ref{Lem=Schwartz}
 	let $\varphi_l: \mathbb{R}^n \rightarrow [0,1], l \in \mathbb{N}_{\geq 0}$ be as in Lemma \ref{Lem=Schwartz} with $c = 2^{-n-2}$.  By Lemma \ref{Lem=Schwartz}  we  have that $2^{-1} \varphi_l h_k$ satisfies   Condition \ref{hormikh}. Let $l_0 \in \mathbb{N}$ be larger than $\max_k \Vert A_k \Vert$ and set $\varphi = \varphi_{l_0}$.

 	Consider the function $\psi_k(s, t) = s_k -t_k$ and $\psi_f(s, t) = f(s) - f(t)$.  	
 	 We have by \eqref{Eqn=Directional} that
 	 \[
 	 \psi_f(\xi)  =  \sum_{k = 1}^n L(h_{k})  \psi_k(\xi) =  \sum_{k = 1}^n L(\varphi_l h_{k})   \psi_k(\xi),
 	 \]
 	  for all $\xi \in \mathbb{R}^n$ with $\xi_i \leq \Vert A_k \Vert$.

 	 	As in the proof of Theorem \ref{Thm=DOIbound}, by Theorem \ref{vector doi thm}   and Theorem   \ref{Thm=JungeMeiInterpolation}  we find through a discretization of $\boldA$ and complex interpolation that for $2 \leq p \leq \infty$,
 	\[
 \Vert	\cI^\boldA_{L(\varphi h_k)}: L_p(\cM) \rightarrow L_p(\cN \otimes \cM) \Vert \leq   c_{n} p.
 	\]
 	So that,
 	\[
 	[f(\boldA), x] = \cI^\boldA_{\psi_f}(x) = 	
 	\sum_{k=1}^n \cI^\boldA_{L(\varphi h_k)} \circ \cI^\boldA_{\psi_k}(x) =
 	\sum_{k=1}^n \cI^\boldA_{L(\varphi h_k)} ( [A_k, x]).
 	\]
 	Then,
 	\[
 	\Vert  [f(\boldA), x] \Vert_p \leq \max_{k}( \Vert \cI_{L(\varphi h_k)}^{\boldA}: L_p \rightarrow L_p \Vert)  \sum_{k=1}^n \Vert [A_k, x]  \Vert_p \leq c_{n} \: p \sum_{k=1}^n \Vert [A_k, x]  \Vert_p.
 	\]
 	This concludes the proof for $2 \leq p < \infty$. For $1 < p \leq 2$ the proof follows by duality just as in Theorem \ref{Thm=DOIbound}. 	
 	
 \end{proof}

 \begin{thm}\label{Thm=VecLip}
 	Let $f: \mathbb{R}^n \rightarrow \cN$ be a $C^{n+3}$-function such that each of the functions $h_k = \partial_k f, k= 1, \ldots, n$ satisfy \eqref{Eqn=SmoothEnough}.
 	There exists a constant $c_{n}$    such that for every $n$-tuples of self-adjoint operators $\boldB = (A_1, \ldots, A_n)$ and $\boldC = (C_1, \ldots, C_n)$ of commuting self-adjoint operators in $\cM$ we have
 	\[
 	\Vert  f(\boldB) - f(\boldC) \Vert_{p} \leq
 	c_{n} \frac{p^2}{p-1} \sum_{k=1}^n \Vert B_k - C_k \Vert_{p}.
 	\]	
 \end{thm}
 \begin{proof}
 	Apply Theorem \ref{Thm=VecCom} to the $n$-tuple  $\left( \begin{array}{cc} B_k & 0 \\ 0 & C_k  \end{array} \right)$ with $k =1, \ldots, n$ and $x =  \left( \begin{array}{cc} 0 & 1 \\ 1 & 0 \end{array} \right)$. See Corollary \ref{Cor=LipEst} for details.
 \end{proof}	

 We apply our results to the particular case that $\cN$ is is an algebra of freely independent semi-circular elements.

 \begin{cor}\label{Cor=SemiCircle}
 	Let $s_i, i \in \mathbb{N}$ be freely independent semi-circular random variables and let $f_i: \mathbb{R} \rightarrow \mathbb{C}$ be $C^{4}$-functions.   Put  $F_l = \sum_{i=1}^l s_i \otimes f_i$ and assume that $F_l$ satisfies \eqref{Eqn=SmoothEnough}.
 	We have for every $l$ that,
 	\[
 	\Vert \sum_{i =1}^l s_i \otimes f_i(B)  - \sum_{i =1}^l s_i \otimes f_i(C) \Vert_p \leq c_{n}   \frac{ p^2}{p-1}  \Vert B - C \Vert_p.
 	\]
 \end{cor}
 \begin{proof}
 	This follows from Theorem \ref{Thm=VecLip} with $n=1$ and  $\boldB = B$ and $\boldA = a$ a single operator and further $f = F_l$.
 \end{proof}

\end{document}